\documentclass[11pt,a4paper,reqno]{amsart}
\usepackage{amsthm,amsmath,amsfonts,amssymb,amsxtra,appendix,bookmark,dsfont,latexsym,bm,hyperref,delarray,color,euscript,amsgen,amsbsy,amsopn,amscd,latexsym,mathrsfs}

\makeatletter

\usepackage{hyperref}

\makeatletter
\def\@tocline#1#2#3#4#5#6#7{\relax
  \ifnum #1>\c@tocdepth 
  \else
    \par \addpenalty\@secpenalty\addvspace{#2}%
    \begingroup \hyphenpenalty\@M
    \@ifempty{#4}{%
      \@tempdima\csname r@tocindent\number#1\endcsname\relax
    }{%
      \@tempdima#4\relax
    }%
    \parindent\z@ \leftskip#3\relax \advance\leftskip\@tempdima\relax
    \rightskip\@pnumwidth plus4em \parfillskip-\@pnumwidth
    #5\leavevmode\hskip-\@tempdima
      \ifcase #1
       \or\or \hskip 1em \or \hskip 2em \else \hskip 3em \fi%
      #6\nobreak\relax
      \dotfill
      \hbox to\@pnumwidth{\@tocpagenum{#7}}
    \par
    \nobreak
    \endgroup
  \fi}
\makeatother

\newtheorem{theorem}{Theorem}
\numberwithin{theorem}{section}
\newtheorem{lemma}[theorem]{Lemma}
\newtheorem{proposition}[theorem]{Proposition}

\newtheorem{corollary}[theorem]{Corollary}

\theoremstyle{definition}

\theoremstyle{remark}
\newtheorem{remark}[theorem]{Remark}

\newcommand{\R}{\mathbb{R}}
\newcommand{\Dtwo}{\mathcal{D}_k}

\newcommand{\bx}{\boldsymbol{x}}
\newcommand{\by}{\boldsymbol{y}}
\newcommand{\ba}{\boldsymbol{a}}
\newcommand{\bu}{\boldsymbol{u}}
\newcommand{\bv}{\boldsymbol{v}}

\newcommand{\be}{\begin{eqnarray*}}
\newcommand{\ee}{\end{eqnarray*}}
\newcommand{\bal}{\begin{align*}}
\newcommand{\ea}{\end{align*}}
\newcommand{\bs}{\boldsymbol}

\newcommand{\Smone}{\mathbb{S}^{m-1}}
\newcommand{\Bm}{\mathbb{B}^{m}}
\newcommand{\C}{\mathbb{C}}

\newcommand{\HK}{\mathcal{H}_k}

\newcommand{\ze}{\bs{\zeta}}
\newcommand{\et}{\bs{\eta}}

\newcommand{\BLK}{\mathcal{B}_l(\R^m\times\Bm,\HK(\C))}
\newcommand{\BBK}{\mathcal{B}_l(\Bm\times\Bm,\HK(\C))}
\newcommand{\BLKS}{\mathcal{B}_l(\Smone\times\Smone,\HK(\Smone))}
\newcommand{\PLK}{\mathcal{P}_l(\R^m\times\Bm,\HK(\C))}
\newcommand{\BPK}{b^p(\Omega\times\Bm,\HK(\C))}

\numberwithin{equation}{section}

\begin{document}

\title[Polynomial Null Solutions, Bosonic Bergman and Hardy Spaces]{Polynomial Null Solutions to Bosonic Laplacians, Bosonic Bergman and Hardy Spaces}

\author[C. Ding]{Chao Ding}
\address{Department of Mathematics and Statistics, Masaryk University, Brno, Czech Republic}
\email{chaoding@math.muni.cz}
\author[P.T. Nguyen]{Phuoc-Tai Nguyen}
\address{Department of Mathematics and Statistics, Masaryk University, Brno, Czech Republic}
\email{ptnguyen@math.muni.cz}
\author[J. Ryan]{John Ryan}
\address{Department of Mathematical Science, University of Arkansas, Fayetteville, AR. U.S.A.}
\email{jryan@uark.edu}

\date{\today}

\begin{abstract}
A bosonic Laplacian, which is a generalization of Laplacian, is constructed as a second order conformally invariant differential operator acting on functions taking values in irreducible representations of the special orthogonal group, hence of the spin group. In this paper, we firstly introduce some properties for homogeneous polynomial null solutions to bosonic Laplacians, which give us some important results, such as an orthogonal decomposition of the space of polynomials in terms of homogeneous polynomial null solutions to bosonic Laplacians, etc. This work helps us to introduce Bergman spaces related to bosonic Laplacians, named as bosonic Bergman spaces, in higher spin spaces. Reproducing kernels for bosonic Bergman spaces in the unit ball and a description of bosonic Bergman projection are given as well. At the end, we investigate bosonic Hardy spaces, which are considered as generalizations of harmonic Hardy spaces. Analogs of some well known results for harmonic Hardy spaces are provided here. For instance, connections to certain complex Borel measure spaces, growth estimates for functions in the bosonic Hardy spaces, etc.
\medskip
	
	
\medskip
	

\end{abstract}

\maketitle

\tableofcontents

\section{Introduction}
In complex analysis, the Bergman and Hardy spaces are particular subspaces of holomorphic functions on the unit disk or the upper-half plane, which play an important role in modern harmonic analysis. The theory of Hardy and Bergman spaces combines techniques from functional analysis, the theory of analytic functions and Lebesgue integration and it has many applications, such as signal processing, Fourier analysis, maximum modulus principle, etc. More details can be found in, for instance, \cite{Duren,DurenB,Fefferman,Heden,Niko,Zhu}.
\par
Here we investigate a function theory related to a particular type of second order conformally invariant differential operators, named as bosonic Laplacians. These second order differential operators act on functions taking values in irreducible representations of $SO(m)$, hence of the spin group $Spin(m)$. Further, for convenience, these representation spaces are usually realized as the spaces of scalar-valued homogeneous harmonic polynomials. Particularly, when the degree of the space of homogeneous harmonic polynomials is zero, the bosonic Laplacian reduces to the classical Laplacian. To introduce all these, we need Clifford analysis and Clifford algebras.
\par
Clifford analysis is considered as not only a higher dimensional function theory offering a generalization of complex analysis but also a refinement of classical harmonic analysis. It is centered around the study of the Dirac operator and monogenic functions (null solutions of the Dirac operator). Hardy and Bergman spaces have also been investigated by many researchers via Clifford analysis in the past decades. For instance, in \cite{Qian}, the authors introduced an analog of Hardy type spaces over a special type of surfaces lying in the conformal closure of $\mathbb{C}^m$ with an application of Vahlen matrices. A technique with homogeneous series expansion was applied in \cite{Bernstein} to study characterizations of certain Dirichlet and Hardy spaces of Clifford-valued monogenic  function in the unit ball. Boundary values of functions in Hardy spaces and applications in signal processing were provided in \cite{Qiansignal}. A frame theory of Hardy spaces was developed by using Cauchy type integral in \cite{Chen}. Hardy spaces related to some perturbed Dirac operators in exterior uniformly rectifiable domains was developed in \cite{Mar}. Reproducing kernel for the module of Clifford valued square-integrable eigenfunctions of the Dirac operator was investigated in \cite{Brackx} and weighted Bergman projections in the monogenic Bergman spaces was studied in \cite{Ren}. More details can be found in, for instance, \cite{Axler,Mitrea,SW1}.
\par
The study of conformally invariant differential operators in the higher spin spaces can be traced back to 1968, when Stein and Weiss \cite{SW} introduced a technique to construct first order conformally invariant differential operators, named as Stein-Weiss gradients, by applying a decomposition of tensor product of two representation spaces of the spin group. It turns out that first order conformally invariant differential operators in the higher spin spaces in Clifford analysis can be constructed as Stein-Weiss gradients as well, see \cite{DWR0}. In \cite{Bures,DeBie,Dunkl, Eelbode}, the first and second order conformally invariant differential operators, named as Rarita-Schwinger operators and the higher spin Laplace operators (also called bosonic Laplacians), in the higher spin spaces via Clifford analysis with algebraic and analytic techniques. Some properties and integral formulas, such as Green's integral formula and Borel-Pompeiu formula, for bosonic Laplacians have been introduced in \cite{DR}. Further, Clerc and \O rsted \cite{CO} introduced a representation-theoretic framework to show the connection between these conformally invariant differential operators in the higher spin spaces and Knapp-Stein intertwining operators. Recently, in \cite{DTR}, the authors solved Dirichlet problems for bosonic Laplacians in the unit ball and upper-half space. Further, they discovered that null solutions to bosonic Laplacians also possess some important properties as the classical Laplacian does, such as the mean-value property, Cauchy's estimates, Liouville's Theorem, etc.
\par
In this paper, we will continue our investigation on analogs of harmonic function theory in higher spin spaces. More specifically, we study homogeneous polynomial null solutions to  bosonic Laplacians, which was firstly described in \cite{DeBie} by a decomposition into irreducible representations of the spin group. In this paper, we provide more information of these particular homogeneous polynomial null solutions, for instance, orthogonality property and a decomposition of certain $L^2$ space on the unit sphere. Analogs of some properties of harmonic Hardy spaces and harmonic Bergman spaces in higher spin spaces in the framework of Clifford analysis are also investigated here. In order to facilitate calculations, Clifford algebras are needed here. The intricate form of the bosonic Laplacians, together
with the rotation action on the second variable and the interaction of the two variables, leads to the invalidity of some classical techniques and highly complicates the analysis.
\\
\textbf{Main results:} 
\begin{enumerate}
\item Homogeneous polynomial null solutions to bosonic Laplacians is studied. An orthogonal decomposition for certain $L^2$ space in terms of these homogeneous polynomial solutions is discovered. We also find that there is an orthogonal property between two homogeneous polynomial null solutions with different degrees, which can be considered as an analog of the orthogonality between spherical harmonics.
\item We define bosonic Bergman spaces in the last section, where we also provide a description for the Bergman reproducing kernel and the bosonic Bergman projection.
\item We introduce bosonic Hardy spaces as generalizations of harmonic Hardy spaces in the higher spin cases. Close relationship between the bosonic Hardy spaces and certain complex Borel measure spaces is provided.
\end{enumerate}
The investigation on bosonic Hardy and Bergman spaces also suggests that many other interesting problems on Hardy and Bergman spaces in classical harmonic analysis can be studied in the higher spin cases as well. For instance, different characterizations of Hardy spaces, the classical Riesz theory on boundary behavior, Berezin transform in Bergman spaces, etc.

\subsection*{Acknowledgements}
Chao Ding and Phuoc-Tai Nguyen are supported by Czech Science Foundation, project GJ19-14413Y.
\section{Preliminaries}
Let $\{\bs{e}_1,\cdots,\bs{e}_m\}$ be a standard orthonormal basis for the $m$-dimensional Euclidean space $\R^m$. Suppose $\bx$ and $\ba$ are two vectors on the unit sphere $\mathbb{S}^{m-1}$, we now show that a reflection of $\bx$ across the hyperplane perpendicular to $\ba$ can simply be expressed as $\ba\bx\ba$. This simplifies the calculation significantly later in this article. To explain this one needs Clifford analysis and Clifford algebras.
\par
The (real) Clifford algebra $\mathcal{C}l_m$ is generated by $\R^m$ with the relationship $\bs{e}_i\bs{e}_j+\bs{e}_j\bs{e}_i=-2\delta_{ij},\ 1\leq i,j\leq m.$ Hence,
an element of the basis of the Clifford algebra can be written in the form $\bs{e}_A=\bs{e}_{j_1}\cdots \bs{e}_{j_r},$ where $A=\{j_1, \cdots, j_r\}\subset \{1, 2, \cdots, m\}$ and $1\leq j_1< j_2 < \cdots < j_r \leq m$. 
Obviously, the $m$-dimensional Euclidean space $\R^m$ can be embedded into $\mathcal{C}l_m$ with the mapping
$\bx=(x_1,\cdots,x_m)\ \mapsto\quad \sum_{j=1}^mx_j\bs{e}_j.$
For $\bx\in\R^m$, one can easily verify that $|\bx|^2=\sum_{j=1}^mx_j^2=-\bx^2$. 
The complex Clifford algebra $\mathcal{C}l_m(\mathbb{C})$ can be realized as $\mathcal{C}l_m(\mathbb{C})=\mathcal{C}l_m\otimes\mathbb{C}$. 
\par
Suppose that $\ba\in \mathbb{S}^{m-1}\subseteq \mathbb{R}^m$, if we consider $\ba\bx\ba$, we may decompose
$$\bx=\bx_{\ba\parallel}+\bx_{\ba\perp},$$
where $\bx_{\ba\parallel}$ stands for the projection of $\bx$ onto $\ba$ and $\bx_{\ba\perp}$ is the rest, perpendicular to $a$. Hence $\bx_{\ba\parallel}$ is a scalar multiple of $\ba$ and we have
$$\ba\bx\ba=\ba\bx_{\ba\parallel}\ba+\ba\bx_{\ba\perp}\ba=-\bx_{\ba\parallel}+\bx_{\ba\perp}.$$
So the action $\ba\bx\ba$ represents a reflection of $\bx$ across the hyperplane perpendicular to $\ba$.
\par
\par
Let $\HK(\C)$ ($1 \leq k \in {\mathbb N}$) be the space of complex-valued homogeneous harmonic polynomials of degree $k$ in $m$-dimensional Euclidean space. If we consider a function $f(\bx,\bu)\in C^{\infty}(\R^m\times\R^m,\HK(\C))$, i.e., for a fixed $\bx\in\R^m$, $f(\bx,\bu)\in\HK(\C)$ with respect to $\bu \in \R^m$. Recall that bosonic Laplacians \cite{Eelbode} are defined as 
\begin{eqnarray}\label{Dtwo}
&&\Dtwo:\ C^{\infty}(\R^m\times\R^m,\HK(\C))\longrightarrow C^{\infty}(\R^m\times\R^m,\HK(\C)),\nonumber\\
&&\Dtwo=\Delta_{\bx}-\frac{4\langle \bu,\nabla_{\bx}\rangle\langle \nabla_{\bu},\nabla_{\bx}\rangle}{m+2k-2}+\frac{4|\bu|^2\langle \nabla_{\bu},\nabla_{\bx}\rangle^2}{(m+2k-2)(m+2k-4)},
\end{eqnarray}
where $\langle\ ,\ \rangle$ is the standard inner product in $\R^m$, $\nabla_{\bx}$ is the gradient with respect to $\bx$. In particular, $
\mathcal{D}_1=\Delta_{\bx}-\frac{4}{m}\langle \bu,\nabla_{\bx}\rangle\langle \nabla_{\bu},\nabla_{\bx}\rangle
$
 is the generalized Maxwell operator. Further, it reduces to the source-free classical Maxwell equations given in terms of the Faraday-tensor when $m=4,\ k=1$ with signature $(-,+,+,+)$. More details can be found in \cite{Eelbode}. A Green's formula for $\Dtwo$ is stated as follows.
\begin{theorem}[Green's formula: scalar-valued version]\cite[Theorem 5]{DR1}\label{Green}
Let $\Omega\subset\R^m$ be an open domain, $\partial\Omega$ is piecewise smooth, and $f,g\in C^2(\Omega\times\Bm,\HK(\R))$. Then, we have
\begin{align*}
&\int\displaylimits_{\Omega}\int\displaylimits_{\Smone}(\Dtwo f(\bx,\bu))g(\bx,\bu)-f(\bx,\bu)(\Dtwo g(\bx,\bu))dS(\bu)d\bx\\
=&\int\displaylimits_{\partial\Omega}\int\displaylimits_{\Smone}(Af)(\bx,\bu)g(\bx,\bu)-f(\bx,\bu)(Ag)(\bx,\bu)dS(\bu)d\sigma(\bx),
\end{align*}
 where $\sigma(x)$ is the area element on $\partial\Omega$ and $n_{\bx}$ is the outward unit normal vector on $\partial\Omega$ and
 \be
 A=\frac{\partial}{\partial n_{\bx}}-\frac{4\langle\bu,n_{\bx}\rangle\langle D_{\bu},D_{\bx}\rangle}{m+2k-2}.
 \ee
\end{theorem} 
 

\section{Polynomials null solutions to bosonic Laplacians}
In \cite{DeBie}, the authors provided a description for homogeneous polynomial null solutions to bosonic Laplacians by applying a decomposition of a tensor product of two representations of the spin group. In this section, we will investigate the space of homogeneous polynomial null solutions to bosonic Laplacians with an analytic approach. More specifically, we firstly study series expansions of solutions to bosonic Laplacians in terms of homogeneous polynomial null solutions in Euclidean spaces. This leads to decompositions of the space of homogeneous polynomials and a certain $L^2$ space in terms of these homogeneous polynomial null solutions in Euclidean spaces. Then, we show that the restriction of these homogeneous polynomial null solutions to the unit sphere gives us analogs of spherical harmonics. A reproducing kernel and an orthogonality property for these homogeneous polynomial null solutions are also introduced here. This section provides us needed tools to investigate bosonic Bergman spaces in the next section.
\par
First, let us introduce some notations. Let $\mathcal{P}(\R^m\times\Bm,\HK(\mathbb{C}))$ be the space of polynomials $f(\bx,\bu)$ such that $f$ is a polynomial in $\bx$ and $\bu$, and for each fixed $\bx\in\R^m$, $f(\bx,\cdot)\in\HK(\mathbb{C})$ with respect to $\bu$. Let $\mathcal{P}_l(\R^m\times\Bm,\HK(\mathbb{C}))$ be the subspace of $\mathcal{P}(\R^m\times\Bm,\HK(\mathbb{C}))$ such that $f$ is a homogeneous polynomial with respect to $\bx$ with degree $l$. We also denote $\BLK=\mathcal{P}_l(\R^m\times\Bm,\HK(\mathbb{C}))\cap\ker\Dtwo$.
\subsection{Real analyticity and homogeneous expansions}
Recall that harmonic functions are real analytic, which means that any harmonic function can be locally expressed as a power series. In this section, we will prove that we also have real analyticity for null solutions to bosonic Laplacians. 
\begin{theorem}\label{analytic}
Let $\Omega\subset\R^m$ be a connected, open bounded domain and $f\in C^2(\Omega\times\Bm,\HK(\mathbb{C}))$. If $\Dtwo f=0$ in $\Omega\times\Bm$, then $f$ is real analytic with respect to $\bx$ in $\Omega$.
\end{theorem}
\begin{proof}
It suffices to show that if $\Dtwo f=0$ in $\Bm\times\Bm$, then $f$ has a power series expansion converging to $f$ in a neighborhood of $0$. The main idea is to use the Poisson integral formula and the series expansion of the Poisson kernel. Recall that if $f(\bx,\bu)\in C^2(\Omega\times\Bm,\HK(\mathbb{C}))$ and $\Dtwo f=0$ in $\Omega\times\Bm$, the Poisson integral formula \cite[Section 5.2]{DTR} is given by
\begin{eqnarray*}
f(\bx,\bv)=\frac{c_{m,k}}{2}\int_{\Smone}\int_{\Smone}\frac{1-|\bx|^2}{|\bx-\ze|^m}Z_k\bigg(\frac{(\bx-\ze)\bu(\bx-\ze)}{|\bx-\ze|^2},\bv\bigg)f(\ze,\bu)dS(\bu)dS(\ze),
\end{eqnarray*}
where $Z_k(\bu,\bv)$ is the reproducing kernel of $k$-homogeneous harmonic polynomials in the following sense
\be
g(\bv)=\int_{\Smone}Z_k(\bu,\bv)g(\bu)dS(\bu),\quad \text{for\ all}\
 g\in\HK(\mathbb{C}).
\ee
\par
On the one hand, we already knew that there is a series expansion, which converges absolutely, for the Poisson kernel of Laplacian $\frac{1-|\bx|^2}{|\bx-\ze|^m}=\sum_{|\bs{\alpha}|=0}^{\infty}p_{\bs{\alpha}}(\ze)\bx^{\bs{\alpha}}$, where $|\bx|<\sqrt{2}-1$ and $\ze\in\Smone$, see \cite[Theorem 1.28]{Axler}. On the other hand, let $\et=\frac{(\bx-\ze)\bu(\bx-\ze)}{|(\bx-\ze)|^2}$, and $\{\varphi_j\}_{j=1}^{d_k}$ be an orthonormal basis for $\HK(\mathbb{C})$, where $d_k=\dim\HK(\mathbb{C})$. Then we have \cite[(5.28)]{Axler} $Z_k(\et,\bv)=\sum_{j=1}^{d_k}\overline{\varphi_j(\et)}\varphi_j(\bv)$. Since $\varphi_j$ is a homogeneous harmonic polynomial of degree $k$, and $\et$ has no singular point for $\bx\in U$, $\ze\in\Smone$, we find that $\varphi_j(\et)$  is analytic for $\bx\in \Bm$. More specifically, the local series expansion of $\varphi_j(\et)$ can be obtained by plugging the series expansion for $\et=\frac{(\bx-\ze)\bu(\bx-\ze)}{|\bx-\ze|^2}$ into $\varphi_j$. Notice that 
\be
\et=\frac{(\bx-\ze)\bu(\bx-\ze)}{|\bx-\ze|^2}=\bu-2\langle\bx-\ze,\bu\rangle|\bx-\ze|^{-2}=\bu-2\langle\bx-\ze,\bu\rangle\sum_{s=0}^{\infty}c_s(|\bx|^2-2\langle\bx,\ze\rangle)^s,
\ee
where $\sum_{s=0}^{\infty}c_s(t-1)^s$ is the Taylor series of $t^{-2}$ on $(0,2)$ at point $t=1$. This series converges to $\et$ absolutely in  $|\bx|<\sqrt{2}-1, \bu\in\overline{\Bm}$ with a similar argument as in \cite[Theorem 1.28]{Axler}. Now, we plug the expansion of $\et$ into $\varphi_j(\et)$ and then rearrange it to obtain $\varphi_j(\et)=\sum_{|\bs{\beta}|=0}^{\infty}q_{\bs{\beta}}(\ze,\bu)\bx^{\bs{\beta}}$.
Now, we have series expansions for $\frac{1-|\bx|^2}{|\bx-\ze|^m}$ and $\varphi_j(\et)$, and both series converge absolutely for $|\bx|<\sqrt{2}-1, \bu,\bv\in\overline{\Bm}$. This allows us to calculate 
\be
&&\frac{1-|\bx|^2}{|\bx-\ze|^m}Z_k(\et,\bv)=\sum_{j=1}^{d_k}\frac{1-|\bx|^2}{|\bx-\ze|^m}\varphi_j(\et)\varphi_j(\bv)\\
&=&\sum_{j=1}^{d_k}\sum_{\bs{|\alpha|=0}}^{\infty}p_{\bs{\alpha}}(\ze)\bx^{\bs{\alpha}}\sum_{|\bs{\beta}|=0}^{\infty}q_{\bs{\beta}}(\ze,\bu)\bx^{\bs{\beta}}\varphi_j(\bv)=\sum_{j=1}^{d_k}\sum_{\bs{|\gamma|=0}}^{\infty}h_{\bs{\gamma}}(\ze,\bu)\bx^{\bs{\gamma}}\varphi_j(\bv)\\
&=&:\sum_{\bs{|\gamma|=0}}^{\infty}g_{\bs{\gamma}}(\ze,\bu,\bv)\bx^{\bs{\gamma}},
\ee
where $h_{\bs{\gamma}}(\ze,\bu)=\sum_{|\bs{\alpha}|+|\bs{\beta}|=|\bs{\gamma}|}p_{\bs{\alpha}}(\ze)q_{\bs{\beta}}(\ze,\bu)$ and $g_{\bs{\gamma}}(\ze,\bu,\bv)$ is a $k$-homogeneous harmonic  polynomial with respect to $\bv$. Further, this series expansion for the Poisson kernel also converges absolutely when $|\bx|<\sqrt{2}-1, \bu,\bv\in\overline{\Bm}$. Now, we plug this series expansion back into the Poisson integral formula to obtain
\begin{align*}
f(\bx,\bv)&=\frac{c_{m,k}}{2}\int_{\Smone}\int_{\Smone}\sum_{\bs{|\gamma|=0}}^{\infty}g_{\bs{\gamma}}(\ze,\bu,\bv)\bx^{\bs{\gamma}}f(\ze,\bu)dS(\bu)dS(\ze)\\
&=\bigg[\frac{c_{m,k}}{2}\int_{\Smone}\int_{\Smone}\sum_{\bs{|\gamma|=0}}^{\infty}g_{\bs{\gamma}}(\ze,\bu,\bv)f(\ze,\bu)dS(\bu)dS(\ze)\bigg]\bx^{\bs{\gamma}}\\
&=\sum_{\bs{|\gamma|=0}}^{\infty}\bigg[\frac{c_{m,k}}{2}\int_{\Smone}\int_{\Smone}g_{\bs{\gamma}}(\ze,\bu,\bv)f(\ze,\bu)dS(\bu)dS(\ze)\bigg]\bx^{\bs{\gamma}}
=:C_{\bs{\gamma}}(\bv)\bx^{\bs{\gamma}},
\end{align*}
where $C_{\bs{\gamma}}(\bv)\in\HK(\bv,\C)$. In the last second step we interchange the integral and infinite sum because of the facts that $\sum_{\bs{|\gamma|=0}}^{\infty}g_{\bs{\gamma}}(\ze,\bu,\bv)\bx^{\bs{\gamma}}$ converges absolutely when $|\bx|<\sqrt{2}-1, \bu,\bv\in\overline{\Bm}$ and $f$ is a polynomial, which allow us to apply Fubini's Theorem.
\end{proof}
The real analyticity for null solutions to bosonic Laplacians allows us to rewrite these solutions as an infinite sum of homogeneous polynomials. This result is crucial for obtaining a decomposition for polynomial null solutions to bosonic Laplacians in the next subsection.
\begin{proposition}\label{homoexp}
Let $f\in C^2(\Omega\times\Bm,\HK(\mathbb{C}))$ and $\Dtwo f=0$ in $\Omega\times\Bm$. Then, given $\bs{a}\in\Omega$, we have $f(\bx,\bu)=\sum_{l=0}^{\infty}f_l(\bx-\bs{a},\bu)$, where $f_l(\bx,\bu)\in\mathcal{B}_l(\Omega\times\Bm,\HK(\mathbb{C}))$.
\end{proposition}
\begin{proof}
Without loss of generality, we assume that $\bs{a}$ is the origin, because for arbitrary $\bs{a}\in \Omega$, one can obtain the result by a translation. We denote
$
f_l(\bx,\bu)=\sum_{|\bs{\alpha}|=l}\frac{\partial^{\bs{\alpha}}f(0,\bu)}{\partial\bx^{\bs{\alpha}}}\bx^{\bs{\alpha}},
$
 and it is easy to see that $f_l\in \HK(\bu,\C)$. From the previous theorem, we know that there exists a neighborhood $U$ of the origin, such that $f(\bx,\bu)=\sum_{l=0}^{\infty}f_l(\bx,\bu)$ for $\bx\in U$.
 \par
 Further, we notice that 
 \begin{eqnarray}\label{df}
 0=\Dtwo f=\Dtwo \sum_{l=0}^{\infty}f_l= \sum_{l=0}^{\infty}\Dtwo f_l.
 \end{eqnarray}
  Since $\Dtwo f_l\in\mathcal{P}_{l-2}(\Omega\times\Bm,\HK(\mathbb{C}))$, which implies that if $\{\varphi_j(\bu)\}_{j=1}^{d_{k}}$ is an orthonormal basis for $\mathcal{H}_{k}(\bu)$, then we can rewrite $\Dtwo f_l(\bx,\bu)=\sum_{j=1}^{d_{k}}a_{j,l-2}(\bx)\varphi_j(\bu)$, where $a_{j,l-2}(\bx)$ is a homogeneous polynomial of $\bx$ with degree $l-2$. Then \eqref{df} gives
  $
  \sum_{l=0}^{\infty}\sum_{j=1}^{d_{k}}a_{j,l-2}(\bx)\varphi_j(\bu)=0.
  $
  However, since $\{\varphi_j(\bu)\}_{j=1}^{d_{k}}$ is an orthonormal basis, we obtain
   $
  \sum_{l=0}^{\infty}a_{j,l-2}(\bx)=0,\ \text{for\ all}\ j,
 $
  which implies that $a_{j,l-2}(\bx)=0$ for all $j,l$. This is equivalent to $\Dtwo f_l=0$ for all $l$, which completes the proof.
\end{proof}

\subsection{A polynomial decomposition for an $L^2$ space on $\Smone\times\Smone$}
Recall that the Poisson integral regarding to bosonic Laplacians is given by
\begin{eqnarray}\label{PI}
P[f](\bx,\bv)=\frac{c_{m,k}}{2}\int_{\Smone}\int_{\Smone}\frac{1-|\bx|^2}{|\bx-\ze|^m}Z_k\bigg(\frac{(\bx-\ze)\bu(\bx-\ze)}{|(\bx-\ze)|^2},\bv\bigg)f(\ze,\bu)dS(\bu)dS(\ze),
\end{eqnarray}
where $f\in\C(\Smone\times\Smone,\HK(\mathbb{C}))$, see \cite[(3.13)]{DTR}. 
Our result below shows that the Poisson integral of a polynomial is a polynomial of a special form.
\begin{proposition}\label{d1}
Let $f(\bx,\bu)\in \mathcal{P}_l(\Bm\times\Bm,\HK(\mathbb{C}))$, then
$P[f\big\vert_{\Smone}]=(1-|\bx|^2)g+f$,
for some polynomial $g(\bx,\bu)$ satisfying $\deg_{\bx}g\leq l-2$ and $g\in\HK(\bu,\C)$ for each fixed $\bx$.  
\end{proposition}
\begin{proof}
When $l=0,1$, the result is obviously true with $g=0$. Now, we assume that $l\geq 2$. Firstly, it is easy to observe that $(1-|\bx|^2)g+f=f$ when $\bx\in\Smone$. Hence,  with the boundary data $f\big\vert_{\Smone}$, if we can find a polynomial $g(\bx,\bu)$ satisfying $\deg_{\bx}g\leq l-2$, $g(\bx,\cdot)\in\HK(\bu,\C)$ for each fixed $\bx$ and $\Dtwo \big((1-|\bx|^2)g+f\big)=0$, i.e., $\Dtwo(1-|\bx|^2)g=-\Dtwo f$, then from the uniqueness of the Dirichlet problem for bosonic Laplacians on the unit ball \cite[Section 4]{DTR}, we immediately have $P[f\big\vert_{\Smone}]=(1-|\bx|^2)g+f$. 
\par
Now, let $W$ stand for the vector space of polynomials $f(\bx,\bu)$ satisfying $\deg_{\bx}f\leq l-2$ and $f(\bx,\cdot)\in\HK(\bu,\C)$ for each fixed $\bx$. We have the following linear map
\be
T:\ W&\longrightarrow& W,\\
             g&\mapsto& \Dtwo(1-|\bx|^2)g.
\ee
Suppose that $T g=0$, which implies that $\Dtwo(1-|\bx|^2)g=0$. However, $(1-|\bx|^2)g=0$ when $\bx\in\Smone$, then according to the Poisson integral formula, $(1-|\bx|^2)g=0$ for all $\bx\in\R^m$, i.e., $g=0$ for all $\bx\in\Bm$. This implies that the linear map $T$ is injective. Since $W$ is finite-dimensional, then the linear map $T$ is also surjective. Therefore, we just proved that given an $f\in  \mathcal{P}_l(\Bm\times\Bm,\HK(\mathbb{C}))$, which implies $\Dtwo f\in W$, there exists a function $g\in W$, such that $\Dtwo(1-|\bx|^2)g=-\Dtwo f$ as desired.
\end{proof}

The following result is crucial in our proof of a direct sum decomposition of certain polynomial spaces below.
\begin{proposition}\label{prop5.4}
Let $f(\bx,\bu)\in\mathcal{P}(\Bm\times\Bm,\HK(\mathbb{C}))$, then $\Dtwo|\bx|^2f\not\equiv 0$ in $\R^m\times\Bm$.
\end{proposition}
\begin{proof}
Suppose that there exists a function $f(\bx,\bu)\in\mathcal{P}(\Bm\times\Bm,\HK(\mathbb{C}))$ with $\deg_{\bx} f= l$ such that $\Dtwo|\bx|^2f= 0$ in $\Bm\times\Bm$. One can see that $|\bx|^2f=f$ when $\bx\in\Smone$. According to the uniqueness for the Dirichlet problem given in \cite[Section 4]{DTR}, we must have $P[f]=|\bx|^2f$ in $\Bm\times\Bm$. However, from the previous proposition, we notice that $\deg_{\bx}P[f]= l$ but $\deg_{\bx}|\bx|^2f=l+2$, which is a contradiction. This completes our proof.
\end{proof}
\begin{remark}
Notice that the functions considered in the above two propositions are homogeneous polynomials with respect to $\bx$ and $\bu$. Therefore, the results above are also true for $(\bx,\bu)\in\R^m\times\R^m$.
\end{remark}
Now, we introduce a direct sum decomposition for a particular polynomial space.
\begin{proposition}\label{polyD}
$\mathcal{P}_l(\R^m\times\Bm,\HK(\mathbb{C}))=\BLK\oplus |\bx|^2\mathcal{P}_{l-2}(\R^m\times\Bm,\HK(\mathbb{C}))$.
\end{proposition}
\begin{proof}
Given an $f\in\mathcal{P}_l(\R^m\times\Bm,\HK(\mathbb{C}))$, with Proposition \ref{d1}, we know that there exists $g\in\mathcal{P}(\R^m\times\Bm,\HK(\mathbb{C}))$ with $\deg_{\bx}g\leq l-2$, such that
\be
f=P[f]+(1-|\bx|^2)g,\ \text{for\ all\ }(\bx,\bu)\in\R^m\times\Bm.
\ee
Now, we take the $l$-homogeneous terms with respect to $\bx$ on the right hand side above to obtain $f=f_l-|\bx|^2g_{l-2}$, where $f_l$ is the $l$-homogeneous term for $P[f]$ and $g_{l-2}$ is the $(l-2)$-homogeneous term of $g$. Further, we know that $\Dtwo f_l=0$ by Proposition \ref{homoexp}. This completes the proof for the existence of the decomposition. To prove the uniqueness for this decomposition, we assume that there exist $f_l,f'_l\in\ker\Dtwo$ and $g_{l-2},g'_{l-2}\in\mathcal{P}_{l-2}(\R^m\times\Bm,\HK(\mathbb{C}))$, such that $f_l-|\bx|^2g_{l-2}=f'_l-|\bx|^2g'_{l-2}$, i.e., $f_l-f'_l=|\bx|^2g_{l-2}-|\bx|^2g'_{l-2}\in\ker\Dtwo$. Then Proposition \ref{prop5.4} tells that $f_l=f'_l$ and $g_{l-2}=g'_{l-2}$, which completes the proof.
\end{proof}
The proposition above immediately provides us a decomposition for $\mathcal{P}_l(\R^m\times\Bm,\HK(\mathbb{C}))$ as follows.
\begin{theorem}\label{decomp}
Every $f\in\mathcal{P}_l(\R^m\times\Bm,\HK(\mathbb{C}))$ can be uniquely written in the form 
\be
f=f_l+|\bx|^2f_{l-2}+\cdots+|\bx|^{2s}f_{l-2s},
\ee
where $f_{l}\in\BLK$ and $s=[\frac{l}{2}]$.
\end{theorem}
From the theorem above, one can also find that
\be
&&\dim\BLK=\dim\PLK-\dim\mathcal{P}_{l-2}(\R^m\times\Bm,\HK(\mathbb{C}))\\
&=&(\dim\mathcal{P}_{l}(\bx)-\dim\mathcal{P}_{l-2}(\bx))\dim\HK(\mathbb{C})=\dim\mathcal{H}_l(\C)\cdot\dim\HK(\mathbb{C}).
\ee
This result coincides the result obtained in \cite[Section 4]{DeBie}. The theorem above is equivalent to 
\be
\PLK&=&\BLK\oplus|\bx|^2\mathcal{B}_{l-2}(\R^m\times\Bm,\HK(\mathbb{C}))\\
&&\oplus\cdots\oplus|\bx|^{l-2s}\mathcal{B}_{l-2s}(\R^m\times\Bm,\HK(\mathbb{C})).
\ee
Later in this paper, we will show that this decomposition is an orthogonal decomposition with respect to the inner product \eqref{inner} below.
\subsection{Spherical homogeneous null solutions to bosonic Laplacians}
In classical harmonic analysis, we know that $L^2(\Smone)=\bigoplus_{k=0}^{\infty} \mathcal{H}_k(\Smone)$. In this subsection, we will introduce an analog for the bosonic Laplacians case.
\par
First, let $L^2(\Smone\times\Smone,\HK(\Smone))$ be the space of functions $f(\bx,\bu)$ with $\bx,\bu\in\Smone$, for each fixed $\bx\in\Smone$, $f(\bx,\bu)\in\HK(\Smone)$ with respect to $\bu$, and
\be
\|f\|_{L^2(\Smone\times\Smone,\HK(\Smone))}:=\bigg(\int_{\Smone}\int_{\Smone}|f(\bx,\bu)|^2dS(\bu)dS(\bx)\bigg)^{\frac{1}{2}}<+\infty.
\ee
Let $\{\varphi_j(\bu)\}_{j=1}^{d_k}$ be an orthonormal basis of $\HK(\bu)$, then for a function $f\in L^2(\Smone\times\Smone,\HK(\Smone))$, it can be written as $f=\sum_{j=1}^{d_k}f_j(\bx)\varphi_j(\bu)$. Given the following inner product for $L^2(\Smone\times\Smone,\HK(\Smone))$
\begin{eqnarray}\label{inner}
\langle f\ |\ g\rangle:=\int_{\Smone}\int_{\Smone}\overline{f(\bx,\bu)}g(\bx,\bu)dS(\bu)dS(\bx),\ f,g\in L^2(\Smone\times\Smone,\HK(\Smone)),
\end{eqnarray}
$L^2(\Smone\times\Smone,\HK(\Smone))$ is indeed a Hilbert space. Now, we introduce a decomposition for this $L^2$ space as below. We remind the reader that we only show that it is a direct sum decomposition at the moment and we will explain it is also an orthogonal decomposition with respect to the inner product $\langle\ |\ \rangle$ later in this paper.
\begin{proposition}\label{sphereD}
A decomposition for $L^2(\Smone\times\Smone,\HK(\Smone))$ is given as follows
\be
L^2(\Smone\times\Smone,\HK(\Smone))=\bigoplus_{l=0}^{\infty}\mathcal{B}_l(\Smone\times\Smone,\HK(\Smone)).
\ee
\end{proposition}
\begin{proof}
Here, we only need to show that $\bigoplus_{l=0}^{\infty}\mathcal{B}_l(\Smone\times\Smone,\HK(\Smone))$ is dense in $L^2(\Smone\times\Smone,\HK(\Smone))$. Let $f\in L^2(\Smone\times\Smone,\HK(\Smone))$ and $f=\sum_{j=1}^{d_k}f_j(\bx)\varphi_j(\bu)$ as explained above. The key to prove this proposition is to observe that $f\in L^2(\Smone\times\Smone,\HK(\Smone))$ is equivalent to $f_j\in L^2(\Smone)$ for all $j=1,\cdots,d_k$. Indeed,
\be
\|f\|^2_{L^2(\Smone\times\Smone,\HK(\Smone))}:&=&\int_{\Smone}\int_{\Smone}\big\vert\sum_{j=1}^{d_k}f_j(\bx)\varphi_j(\bu)\big\vert^2dS(\bu)dS(\bx)\\
&=&\sum_{j=1}^{d_k}\int_{\Smone}|f_j(\bx)|^2dS(\bx),
\ee
due to $\{\varphi_j\}_{j=1}^{d_k}$ is an orthonormal basis for $\HK(\Smone)$. As in classical harmonic analysis, there exists a sequence of polynomials $\{p_{l,j}\}_{l=1}^{\infty}$ such that $\|p_{l,j}-f_j\|_{L^2(\Smone)}\rightarrow 0$ as $l\rightarrow \infty$ for all $j$. This is equivalent to 
\begin{eqnarray}\label{w1}
\int_{\Smone}\int_{\Smone}\bigg\vert\sum_{j=1}^{d_k}p_{l,j}(\bx)\varphi_j(\bu)-\sum_{j=1}^{d_k}f_j(\bx)\varphi_j(\bu)\bigg\vert^2dS(\bu)dS(\bx)\rightarrow 0
\end{eqnarray}
 as $l\rightarrow \infty$. Further, Theorem \ref{decomp} tells us that for each $l$, $\sum_{j=1}^{d_k}p_{l,j}(\bx)\varphi_j(\bu)\in\mathcal{P}(\Smone\times\Smone,\HK(\Smone))$ with $\deg_{\bx}\leq l$, so that there exists a sequence $\{q^l_s\}_{s=0}^{\infty}$, where $q^l_s\in\mathcal{B}_s(\Smone\times\Smone,\HK(\Smone))$ such that
$
\sum_{j=1}^{d_k}p_{l,j}(\bx)\varphi_j(\bu)=\sum_{s=0}^{[\frac{l}{2}]}q^l_s(\bx,\bu).
$
Combining this with \eqref{w1}, we have $\|\sum_{s=0}^{[\frac{l}{2}]}q^l_s-f\|_{L^2(\Smone\times\Smone,\HK(\Smone))}\rightarrow 0$ as $l\rightarrow \infty$, which completes the proof.
\end{proof}
\begin{remark}
From the proof above, one might notice that $\bigoplus_{l=0}^{\infty}\mathcal{H}_l(\Smone)\times\HK(\Smone)$ is also dense in $L^2(\Smone\times\Smone,\HK(\Smone))$, where $\mathcal{H}_l(\Smone)$ is with respect to $\bx$ and $\HK(\Smone)$ is with respect to $\bu$. Here, we provides a different decomposition for $L^2(\Smone\times\Smone,\HK(\Smone))$ related to $\Dtwo$.
\end{remark}
Recall that $L^2(\Smone\times\Smone,\HK(\Smone))$ is a Hilbert space with the given inner product in \eqref{inner}, since $\BLKS$ is a finite dimensional inner product subspace of $L^2(\Smone\times\Smone,\HK(\Smone))$， there exists a unique function $J_{l,k}(\cdot,\bx,\cdot,\bv)\in \BLKS$, such that
\begin{eqnarray}\label{JLK}
f(\bx,\bv)=\langle J_{l,k}(\cdot,\bx,\cdot,\bv)\ |\ f\rangle= \int_{\Smone}\int_{\Smone}\overline{J_{l,k}(\ze,\bx,\bu,\bv)}f(\ze,\bu)dS(\bu)dS(\ze),
\end{eqnarray}
for all $f\in\BLK$. With a similar argument as in \cite[Proposition 5.27]{Axler}, one can easily obtain similar basic properties for the reproducing kernel $J_{l,k}(\cdot,\bx,\cdot,\bv)$ as follows.
\begin{proposition}\label{propR}
Suppose $\ze,\bx,\bu,\bv\in\Smone$ and $l,k\geq 0$. Then, we have
\begin{enumerate}
\item $J_{l,k}$ is real valued.
\item $J_{l,k}(\ze,\bx,\bu,\bv)=J_{l,k}(\bx,\ze,\bv,\bu)$.
\item$J_{l,k}(T(\ze),\bx,T(\bu),\bv)=J_{l,k}(\ze,T^{-1}(\bx),\bu,T^{-1}(\bv))$, for all $T\in O(m)$.
\item $J_{l,k}(\bx,\bx,\bu,\bu)=\dim\BLK$.
\item $|J_{l,k}(\ze,\bx,\bu,\bv)|\leq\dim\BLK$.
\end{enumerate}
\end{proposition}
\begin{proof}
Here we only provide an outline proof. To prove $(1)$, we assume $f\in\BLKS$ is real-valued. Then
\begin{align*}
0&=\mathrm{Im}f(\bx,\bv)=\mathrm{Im}\int_{\Smone}\int_{\Smone}\overline{J_{l,k}(\ze,\bx,\bu,\bv)}f(\ze,\bu)dS(\ze)dS(\bu)\\
&=-\int_{\Smone}\int_{\Smone}\mathrm{Im}J_{l,k}(\ze,\bx,\bu,\bv)f(\ze,\bu)dS(\ze)dS(\bu).
\end{align*}
Now, we let $f=\mathrm{Im}J_{l,k}(\ze,\bx,\bu,\bv)$ to immediately obtain $\mathrm{Im}J_{l,k}=0$.
\par
To prove $(2)$, let $\{\phi_1,\cdots,\phi_{d_{l,k}}\}$ be an orthonormal basis of $\BLKS$, where $d_{l,k}=\dim\BLKS$. Then, one has
\begin{eqnarray}\label{Jinv}
J_{l,k}(\ze,\bx,\bu,\bv)=\sum_{j=1}^{d_{l,k}}\langle \phi_j\ |\ J_{l,k}(\cdot,\bx,\cdot,\bv) \rangle\phi_j(\ze,\bu)=\sum_{j=1}^{d_{l,k}}\overline{\phi_j(\bx,\bv)}\phi_j(\ze,\bu).
\end{eqnarray}
Since $J_{l,k}$ is real-valued, the equation above is unchanged after taking complex conjugation, which implies $(2)$. One can obtain $(3)$ immediately from the definition of the reproducing kernel and the fact that the surface area element $dS$ is rotationally invariant.
\par
In $(3)$, if we let $\bx=T(\ze)$ and $\bv=T(\bu)$ in $J_{l,k}(T(\ze),\bx,T(\bu),\bv)$, then we obtain $J_{l,k}(T(\ze),T(\ze),T(\bu),T(\bu))=J_{l,k}(\ze,\ze,\bu,\bu)$, i.e., $J_{l,k}(\ze,\ze,\bu,\bu)$ is invariant under rotation. Letting $\bx=\ze$ and $\bv=\bu$ in \ref{Jinv} immediately gives us $(4)$. Taking absolute value on both sides of \ref{Jinv} and applying Cauchy-Schwarz inequality can easily provide us $(5)$.
\end{proof}
Now, with the Green's formula given in Theorem \ref{Green}, we claim that there is an orthogonality property between homogeneous polynomial null solutions to $\Dtwo$. More specifically,
\begin{theorem}\label{thmortho}
Suppose that $f\in\mathcal{B}_s(\R^m\times\Bm,\mathcal{H}_k(\C))$ and $g\in\mathcal{B}_t(\R^m\times\Bm,\mathcal{H}_l(\C))$, then $f$ is orthogonal to $g$ with respect to the inner product given in \eqref{inner} when $s\neq t$ or $k\neq l$.
\end{theorem}
\begin{proof}
Firstly, it is easy to observe that $\langle f\ |\ g\rangle=0$ when $k\neq l$, this is because of the already known orthogonality between functions in $\HK(\C)$ and functions in $\mathcal{H}_l(\C)$ with respect to the $L^2$ inner product on $\Smone$. Hence, we assume $k=l$ in the rest of the proof. Recall that the Green's formula in Theorem \ref{Green} tells us that
\begin{align}\label{Geqn}
\int\displaylimits_{\Smone}\int\displaylimits_{\Smone}(Af)(\bx,\bu)g(\bx,\bu)-f(\bx,\bu)(Ag)(\bx,\bu)dS(\bu)dS(\bx)=0,
\end{align}
 where
 \begin{align*}
 A=\frac{\partial}{\partial n_{\bx}}-\frac{4\langle\bu,{\bx}\rangle\langle D_{\bu},D_{\bx}\rangle}{m+2k-2}.
 \end{align*}
 Suppose $\bx=r\boldsymbol{\zeta}\in\R^m$, $\bu=\lambda\boldsymbol{\omega}\in\R^m$,  we notice that $r=\sqrt{x_1^2+\cdots+x_m^2}$, then $\frac{\partial }{\partial x_i}=\frac{x_i}{r}\frac{\partial}{\partial r}$. Similarly, we have $\frac{\partial }{\partial u_i}=\frac{u_i}{\lambda}\frac{\partial}{\partial \lambda}$. Then we have
 \begin{align*}
 &Af(\bx,\bu)\bigg\vert_{r,\lambda=1}=\bigg(\frac{\partial}{\partial n_{\bx}}-\frac{4\langle\bu,{\bx}\rangle\langle D_{\bu},D_{\bx}\rangle}{m+2k-2}\bigg)f(r\boldsymbol{\zeta},\lambda\boldsymbol{\omega})\bigg\vert_{r,\lambda=1}\\
 =&\bigg(\frac{\partial}{\partial r}-\frac{4\langle\bu,\bx\rangle^2}{(m+2k-2)r\lambda}\frac{\partial^2}{\partial r\partial\lambda}\bigg)\bigg\vert_{r,\lambda=1}f(r\boldsymbol{\zeta},\lambda\boldsymbol{\omega})=\bigg(s-\frac{4sk\langle\boldsymbol{\zeta},\boldsymbol{\omega}\rangle^2}{m+2k-2}\bigg)f(\boldsymbol{\zeta},\boldsymbol{\omega}).
 \end{align*}
 Similarly, we have
  \begin{align*}
 &Ag(\bx,\bu)\bigg\vert_{r,\lambda=1}=\bigg(t-\frac{4tk\langle\boldsymbol{\zeta},\boldsymbol{\omega}\rangle^2}{m+2k-2}\bigg)g(\boldsymbol{\zeta},\boldsymbol{\omega}).
 \end{align*}
 Therefore, equation \eqref{Geqn} becomes
 \begin{align*}
 0=\int_{\Smone}\int_{\Smone}(s-t)\bigg(1-\frac{4k\langle\boldsymbol{\omega},\boldsymbol{\zeta}\rangle^2}{m+2k-2}\bigg)f(\boldsymbol{\zeta},\boldsymbol{\omega})g(\boldsymbol{\zeta},\boldsymbol{\omega})dS(\boldsymbol{\omega})dS(\boldsymbol{\zeta}).
 \end{align*}
 Since the equation above holds for all $f\in\mathcal{B}_s(\R^m\times\Bm,\mathcal{H}_k(\C))$ and $g\in\mathcal{B}_t(\R^m\times\Bm,\mathcal{H}_k(\C))$, we immediately have $s=t$, which completes the proof.
\end{proof}
\begin{remark}
The theorem above implies that the decompositions given in Proposition \ref{polyD}, \ref{decomp} and \ref{sphereD} are all  orthogonal decompositions. Further, with Proposition \ref{polyD}, the theorem above immediately gives us the following useful corollary.
\end{remark}
\begin{corollary}\label{CorOrtho}
Suppose $f\in\mathcal{P}_s(\Smone\times\Smone,\HK(\Smone))$ and $g\in\mathcal{B}_t(\Smone\times\Smone,\HK(\Smone))$. Then, $f$ is orthogonal to $g$ with respect to the inner product \eqref{inner} when $s<t$.
\end{corollary}

Our next result shows that the Poisson integral can be expressed in terms of the reproducing kernel $J_{l,k}(\ze,\bx,\bu,\bv)$.
\begin{proposition}\label{PropPoly}
Suppose that $f\in\mathcal{P}_l(\R^m\times\Bm,\mathcal{H}_k(\C))$. Then, we have
\begin{align}\label{aoe}
P[f\vert_{\Smone}](\bx,\bv)=\sum_{s=0}^l\int_{\Smone}\int_{\Smone}\overline{J_{s,k}(\ze,\bx,\bu,\bv)}f(\ze,\bu)dS(\bu)dS(\ze)
\end{align}
for every $\bx,\bv\in\Bm$.
\end{proposition}
\begin{proof}
By Proposition \ref{PI}, $P[f\vert_{\Smone}]$ is a polynomial of degree at most $l$ of $\bx$ and hence can be written in the form
\begin{align*}
P[f\vert_{\Smone}]=\sum_{s=0}^lf_s,
\end{align*}
where $f_s\in \mathcal{B}_s(\R^m\times\Bm,\mathcal{H}_k(\C))$. For each $\bx,\bv\in\Bm$ and each $s$ we have
\begin{align*}
f_s(\bx,\bv)=&\int_{\Smone}\int_{\Smone}\overline{J_{s,k}(\ze,\bx,\bu,\bv)}f_s(\ze,\bu)dS(\ze)dS(\bu)\\
=&\int_{\Smone}\int_{\Smone}\overline{J_{s,k}(\ze,\bx,\bu,\bv)}\sum_{j=0}^lf_j(\ze,\bu)dS(\ze)dS(\bu)\\
=&\int_{\Smone}\int_{\Smone}\overline{J_{s,k}(\ze,\bx,\bu,\bv)}f(\ze,\bu)dS(\ze)dS(\bu),
\end{align*} 
where the first equality comes from the Poisson integral formula, the second equality comes from the orthogonality property given in Corollary \ref{CorOrtho}, and the third equality holds since $f$ and its Poisson integral $\sum_{j=0}^lf_j(\ze,\bu)$ agree on $\Smone\times\Smone$. Combining the last equation with \eqref{aoe} gives the desired result.
\end{proof}
The following result tells us that the expression for the reproducing kernel $J_{l,k}$ can be obtained from a homogeneous series expansion of the Poisson kernel.
\begin{proposition}\label{PJ}
For $m\geq 2$, we have
\be
P(\ze,\bx,\bu,\bv)=\sum_{l=0}^{\infty}J_{l,k}(\ze,\bx,\bu,\bv)
\ee
for all $\bx,\bv\in\Bm,\ \ze,\bu\in\Smone$. The series converges absolutely and uniformly on $\Smone\times  K\times\Smone\times\overline{\Bm}$, where $K\subset\Bm$ is a sufficiently small compact subset containing $0$.
\end{proposition}
\begin{proof}
Let $\bx=r\et$ and $\bv=\lambda\bs{\gamma}$ with $\et,\bs{\gamma}\in\Smone$. Then, with the statement $(5)$ in Proposition \ref{propR},  one can easily see that the series $\sum_{l=0}^{\infty}J_{l,k}(\ze,\bx,\bu,\bv) $ converges absolutely and uniformly on $\Smone\times  K\times\Smone\times\overline{\Bm}$, where $K\subset\Bm$ is a sufficiently small compact subset containing $0$. Now suppose that $f\in\mathcal{P}(\Smone\times\Smone,\HK(\mathbb{C}))$, with Poposition \ref{PropPoly} and the orthogonal properties given in Corollary \ref{CorOrtho}, we have
\begin{align*}
&\int_{\Smone}\int_{\Smone}P(\ze,\bx,\bu,\bv)f(\ze,\bu)dS(\bu)dS(\ze)\\
=&\int_{\Smone}\int_{\Smone}\sum_{l=0}^{\infty}J_{l,k}(\ze,\bx,\bu,\bv)f(\ze,\bu)dS(\bu)dS(\ze).
\end{align*}
Since $\mathcal{P}(\Smone\times\Smone,\HK(\mathbb{C}))$ is dense in $L^2(\Smone\times\Smone,\HK(\mathbb{C}))$, this already implies that 
\be
P(\ze,\bx,\bu,\bv)=\sum_{l=0}^{\infty}J_{l,k}(\ze,\bx,\bu,\bv),
\ee
which completes the proof.
\end{proof}
\begin{remark}
We have already noticed the fact that functions in $\BLK$ are homogeneous polynomials in the variables $\bx$ and $\bu$, and these functions are uniquely determined by their values on the unit sphere (or arbitrarily small ball with center $0$). Therefore, although the expression of $J_{l,k}$ is obtained in a sufficiently small ball centered at $0$, we immediately know that it is also the expression for $J_{l,k}$ in $\R^m\times\R^m$. 
\end{remark}

Proposition \ref{PJ} also implies that for any $f\in C^2(\Bm\times\Bm,\HK(\C))$ and $\Dtwo f=0$, the homogeneous series expansion for $f$ given in Proposition \ref{homoexp} has a stronger convergence property as follows.
\begin{proposition}\label{PropConverge}
If $f\in C^2(B(\bs{a},r)\times\Bm,\HK(\C))$ and $\Dtwo f=0$ in $B(\bs{a},r)\times\Bm$. Then there exist $f_l\in\mathcal{B}_l(\R^m\times\Bm,\HK(\C))$ such that
$
f(\bx,\bv)=\sum_{l=0}^{\infty}f_l(\bx-\bs{a},\bv)
$
for all $\bx\in B(\bs{a},r)$ and $\bv\in\Bm$. Further, the series converges absolutely and uniformly on $K\times\overline{\Bm}$, where $K$ is a compact subset of $ B(\bs{a},r)$.
\end{proposition}
\begin{proof}
Firstly, we assume that $f\in C^2(\Bm\times\Bm,\HK(\C))\cap C(\overline{\Bm}\times\overline{\Bm},\HK(\C))$ and $\Dtwo f=0$ in $\Bm\times\Bm$. Then, with the Poisson integral formula and Proposition \ref{PJ}, for $\bx\in B(0,\epsilon),\ \bv\in\Bm$, with $0<\epsilon<1$, we have
\bal
f(\bx,\bv)=&\frac{c_{m,k}}{2}\int_{\Smone}\int_{\Smone}P(\ze,\bx,\bu,\bv)f(\ze,\bu)dS(\bu)dS(\ze)\\
=&\frac{c_{m,k}}{2}\int_{\Smone}\int_{\Smone}\sum_{l=0}^{\infty}J_{l,k}(\ze,\bx,\bu,\bv)f(\ze,\bu)dS(\bu)dS(\ze)\\
=&\frac{c_{m,k}}{2}\sum_{l=0}^{\infty}\int_{\Smone}\int_{\Smone}J_{l,k}(\ze,\bx,\bu,\bv)f(\ze,\bu)dS(\bu)dS(\ze)
=\sum_{l=0}^{\infty}f_l(\bx,\bv),
\end{align*}
where $f_l(\bx,\bv)=\int_{\Smone}\int_{\Smone}J_{l,k}(\ze,\bx,\bu,\bv)f(\ze,\bu)dS(\bu)dS(\ze)$ and $\Dtwo f_l=0$ in $\R^m\times\Bm$. Notice that $(5)$ of Proposition \ref{propR} tells us that 
\begin{align*}
|J_{l,k}(\ze,\bx,\bu,\bv)|&\leq\dim\BLK |\bx|^l|\bv|^k=\dim\HK\cdot\dim\mathcal{H}_l|\bx|^l|\bv|^k\\
&\leq C_m(kl)^{m-2}|\bx|^l|\bv|^k,
\end{align*}
 where $C_m$ is a constant only depending on $m$. Hence, one can see that 
\be
\sum_{l=0}^{\infty}|f_l(\bx,\bv)|\leq C_mk^{m-2}|\bv|^k\sum_{l=0}^{\infty}\int_{\Smone}\int_{\Smone}l^{m-2}|\bx|^l|f(\ze,\bu)|dS(\bu)dS(\ze)<+\infty,
\ee
when $(\bx,\bv)\in \overline{B(0,\epsilon)}\times\overline{\Bm}$. This implies that $\sum_{l=0}^{\infty}f_l$ converges absolutely and uniformly to $f$ in $\overline{B(0,\epsilon)}\times\overline{\Bm}$. Applying a dilation and a translation to the argument above can immediately give us the result on $B(\bs{a},r)\times\Bm$ as desired.
\end{proof}


\section{Bergman spaces related to bosonic Laplacians}
Let $\Omega$ be an open bounded domain in $\R^m$ and $1\leq p<\infty$. In this section, we will introduce bosonic Bergman spaces, denoted by $b^p(\Omega\times\Bm,\HK(\C))$, which are generalizations of harmonic Bergman spaces in higher spin spaces. It turns out that this bosonic Bergman space is also a Hilbert space when $p=2$ with respect to a given $L^2$ inner product. This reveals the existence of a reproducing kernel for $b^2(\Omega\times\Bm,\HK(\C))$, and then a description for this reproducing kernel and a related Bergman projection is provided in terms of the reproducing kernel $J_{l,k}$ \eqref{JLK} when $\Omega=\Bm$.
\par
The bosonic Bergman space $b^p(\Omega\times\Bm,\HK(\C))$ is the set of functions $f(\bx,\bu)\in L^p(\Omega\times\Bm,\HK(\C))\cap C^2(\Omega\times\Bm,\HK(\C))$ satisfying $\Dtwo f=0$ in $\Omega\times\Bm$ and
\be
\|f\|_{b^p(\Omega\times\Bm,\HK(\C))}:=\bigg(\int_{\Omega}\int_{\Bm}|f(\bx,\bu)|^pd\bu d\bx\bigg)^{\frac{1}{p}}<+\infty.
\ee
\subsection{Reproducing kernels for bosonic Bergman spaces}
For a fixed $(\bx,\bu)\in\Omega\times\Bm$, we call the linear map $f\mapsto f(\bx,\bu)$ the point evaluation of $f$ at $(\bx,\bu)$. The following proposition shows that point evaluation is continuous on $b^p(\Omega\times\Bm,\HK(\C))$.
\begin{proposition}\label{BergmanPointEv}
Suppose $f\in b^p(\Omega\times\Bm,\HK(\C))$, $\bs{a}\in\Omega$ and $\bv\in\Bm$. Then,
\be
|f(\bs{a},\bv)|\leq\frac{(m+2k-2)\|f\|_{\BPK}}{(m-2)V(\Bm)^{2/p}d(\bs{a},\partial\Omega)^{m/p}d(\bv,\Smone)^{m/p}}.
\ee
\end{proposition}
\begin{proof}
Let $0<r_1<d(\bs{a},\partial\Omega)$ and $0<r_2<d(\bv,\Smone)$. We firstly apply the volume version of the mean-value property \cite[Proposition 5.3]{DTR} with respect to $\bx$ on $B(\bx,r_1)$ to obtain
\begin{eqnarray}\label{pointev1}
|f(\bs{a},\bv)|^p\leq\bigg(\frac{m+2k-2}{m-2}\bigg)^{p}V(B(\bs{a},r_1))^{-1}\int_{B(\bs{a},r_1)} |f(\bx,\bs{\omega})|^pd\bx,
\end{eqnarray}
where $\bs{\omega}=\frac{(\bx-\bs{a})\bv(\bx-\bs{a})}{|\bx-\bs{a}|^2}$ and Jensen's inequality is applied above. Further, we notice that $\bs{\omega}$ is obtained from $\bv$ by a rotation, which implies that $f(\bx,\bs{\omega})\in\HK(\bs{\omega},\C)$. This motivates us to apply the volume version of the mean-value property to $f(\bx,\bs{\omega})$ with respect to $\bs{\omega}$ on $B(\bs{\omega},r_2)$ to have
\begin{eqnarray}\label{pointev2}
|f(\bx,\bs{\omega})|^p\leq V(B(\bs{\omega},r_2))^{-1}\int_{B(\bs{\omega},r_2)}|f(\bx,\bu)|^pd\bu.
\end{eqnarray}
Plugging \eqref{pointev2} into \eqref{pointev1}, we obtain
\begin{align*}
|f(\bs{a},\bv)|^p&\leq\bigg(\frac{m+2k-2}{m-2}\bigg)^{p}V(B(\bs{a},r_1))^{-1}\int_{B(\bs{a},r_1)}V(B(\bs{\omega},r_2))^{-1}\int_{B(\bs{\omega},r_2)}|f(\bx,\bu)|^pd\bu d\bx\\
&\leq \bigg(\frac{m+2k-2}{m-2}\bigg)^{p}r_1^{-m}r_2^{-m}V(\Bm)^{-2}\|f\|_{\BPK}^p.
\end{align*}
Now, we let $r_1\rightarrow d(\bs{a},\partial\Omega)$, $r_2\rightarrow d(\bv,\Smone)$ and take $p$th root on both sides above to complete the proof.
\end{proof}
With the previous proposition and a similar argument as in Proposition \ref{hpbanach}, we obtain the following result, which tells us that $\BPK$ is a Banach space.
\begin{proposition}
The bosonic Bergman space $\BPK$ is a closed subspace of $L^p(\Omega\times\Bm,\HK(\C))$.
\end{proposition}
In particular, let $p=2$, we have that $b^2(\Omega\times\Bm,\HK(\C))$ is a Hilbert space with inner product 
\begin{eqnarray*}\label{innerb2}
\langle f, g\rangle_{b^2}:=\int_{\Omega}\int_{\Bm}\overline{f(\bx,\bu)}g(\bx,\bu)d\bu d\bx,\ f,g\in L^2(\Omega\times\Bm,\HK(\C)).
\end{eqnarray*}
Now, for each fixed $\bx\in\Omega$ and $\bv\in\Bm$, with Proposition \ref{BergmanPointEv}, we notice that the map $f\mapsto f(\bx,\bv)$ is a bounded linear functional on the Hilbert space $b^2(\Omega\times\Bm,\HK(\C))$. Hence, there exists a unique function $R_{k,\Omega\times\Bm}(\cdot,\bx,\cdot,\bv)\in b^2(\Omega\times\Bm,\HK(\C))$, such that 
\be
f(\bx,\bv)=\langle R_{k,\Omega\times\Bm}(\cdot,\bx,\cdot,\bv),f\rangle_{b^2}=\int_{\Omega}\int_{\Bm}\overline{R_{k,\Omega\times\Bm}(\by,\bx,\bu,\bv)}f(\by,\bu)d\bu d\by.
\ee
We call the function $R_{k,\Omega\times\Bm}$ the \emph{reproducing kernel} of $b^2(\Omega\times\Bm,\HK(\C))$. One can also obtain similar properties for $R_{k,\Omega\times\Bm}$ with a similar proof as in Proposition \ref{propR}.
\begin{proposition}
The reproducing kernel $R_{k,\Omega\times\Bm}$ has the following properties.
\begin{enumerate}
\item $R_{k,\Omega\times\Bm}$ is real valued.
\item $R_{k,\Omega\times\Bm}(\by,\bx,\bu,\bv)=R_{k,\Omega\times\Bm}(\bx,\by,\bv,\bu)$ for all $\bx,\by\in\Omega$ and $\bu,\bv\in\Bm$.
\item If $\{\phi_j\}_{j=1}^{\infty}$ is an orthonormal basis for $b^2(\Omega\times\Bm,\HK(\C))$, then
\be
R_{k,\Omega\times\Bm}(\by,\bx,\bu,\bv)=\sum_{j=1}^{\infty}\overline{\phi_j(\bx,\bv)}\phi_j(\by,\bu).
\ee
\item $\|R_{k,\Omega\times\Bm}(\cdot,\bx,\cdot,\bv)\|_{b^2(\Omega\times\Bm,\HK(\C))}=\sqrt{R_{k,\Omega\times\Bm}(\bx,\bx,\bv,\bv)}$ for all $\bx\in\Omega$ and $\bv\in\Bm$.
\end{enumerate}
\end{proposition}
Since $b^2(\Omega\times\Bm,\HK(\C))$ is a closed subspace of the Hilbert space $L^2(\Omega\times\Bm,\C)$, there is a unique orthogonal projection of $L^2(\Omega\times\Bm,\C)$ onto $b^2(\Omega\times\Bm,\HK(\C))$, denoted by $B_{k,\Omega\times\Bm}$. We call this projection the bosonic Bergman projection on $\Omega\times\Bm$. The following proposition reveals the connection between the bosonic Bergman projection and the reproducing kernel $R_{k,\Omega\times\Bm}$.
\begin{proposition}\label{BergmanProjection}
Suppose $\bx\in\Omega$ and $\bv\in\Bm$, then we have
\be
B_{k,\Omega\times\Bm}[f](\bx,\bv)=\int_{\Omega}\int_{\Bm}R_{k,\Omega\times\Bm}(\by,\bx,\bu,\bv)f(\by,\bu)d\bu d\by,
\ee
for all $f\in L^2(\Omega\times\Bm,\C)$.
\end{proposition}
\begin{proof}
Let $f\in L^2(\Omega\times\Bm,\C)$, $\bx\in\Omega$ and $\bv\in\Bm$. Then
\begin{align*}
B_{k,\Omega\times\Bm}[f](\bx,\bv)=\langle R_{k,\Omega\times\Bm}(\cdot,\bx,\cdot,\bv),B_{k,\Omega\times\Bm}[f]\rangle_{b^2}.
\end{align*}
Since $B_{k,\Omega\times\Bm}$ is an orthogonal projection, it is also self-adjoint and $R_{k,\Omega\times\Bm}(\bx,\cdot,\bv,\cdot)\in b^2(\Omega\times\Bm,\HK(\C))$ tells us that $B_{k,\Omega\times\Bm}[R_{k,\Omega\times\Bm}(\cdot,\bx,\cdot,\bv)]=R_{k,\Omega\times\Bm}(\cdot,\bx,\cdot,\bv)$. Therefore, the equation above is equal to
\be
\langle R_{k,\Omega\times\Bm}(\cdot,\bx,\cdot,\bv),f\rangle_{b^2}=\int_{\Omega}\int_{\Bm}\overline{R_{k,\Omega\times\Bm}(\by,\bx,\bu,\bv)}f(\by,\bu)d\bu d\by,
\ee
which completes the proof since $R_{k,\Omega\times\Bm}$ is real-valued.
\end{proof}
\subsection{Reproducing kernels on the unit ball}
In this section, we will introduce the connection between the reproducing kernel $R_{k,\Bm\times\Bm}$ and the reproducing kernel $J_{l,k}$ of $\BLK$ (Section $4.3$). This also provides an expression of $B_{k,\Bm\times\Bm}$ in terms of $J_{l,k}$.
\par
Recall that in \eqref{JLK}, we define the reproducing kernel $J_{l,k}$ with integrals over the unit sphere. Now, we will firstly use polar coordinates to obtain an analog of \eqref{JLK} with integrals over the unit ball for $f\in\BBK$.
\begin{align}\label{BtoS}
& \int_{\Bm}\int_{\Bm}\overline{J_{l,k}(\by,\bx,\bu,\bv)}f(\by,\bu)d\bu d\by\nonumber\\
=&\int_0^1\int_{0}^1\int_{\Smone}\int_{\Smone}r_1^{m-1}r_2^{m-1}\overline{J_{l,k}(r_1\ze,\bx,r_2\et,\bv)}f(r_1\ze,r_2\et)dS(\ze)dS(\et)dr_1dr_2\nonumber\\
=&\int_0^1\int_{0}^1\int_{\Smone}\int_{\Smone}r_1^{m+2l-1}r_2^{m+2k-1}\overline{J_{l,k}(\ze,\bx,\et,\bv)}f(\ze,\et)dS(\ze)dS(\et)dr_1dr_2\nonumber\\
=&(m+2k)^{-1}(m+2l)^{-1}\int_{\Smone}\int_{\Smone}\overline{J_{l,k}(\ze,\bx,\et,\bv)}f(\ze,\et)dS(\ze)dS(\et)\nonumber\\
=&(m+2k)^{-1}(m+2l)^{-1}f(\bx,\bv),
\end{align}
where $\ze=\frac{\by}{|\by|}$ and $\et=\frac{\bu}{|\bu|}$. Now, we claim that all homogeneous polynomial null solutions in $\mathcal{P}(\Bm\times\Bm,\HK(\C))$ for $\Dtwo$ are dense in $b^2(\Bm\times\Bm,\HK(\C))$. In other words,
\begin{proposition}
There holds
\be
b^2(\Bm\times\Bm,\HK(\C))=\bigoplus_{l=0}^{\infty}\mathcal{B}_l(\Bm\times\Bm,\HK(\C)).
\ee
\end{proposition}
\begin{proof}
We firstly notice that for a function $f\in L^2(\Bm\times\Bm,\HK(\C))$, $f_l(\bx,\bv):=f(\frac{l-1}{l}\bx,\bv)\in L^2(\Bm\times\Bm,\HK(\C))$ for $l=1,2,\cdots$. Further, we can see that $f_l$ converges to $f$ in $L^2(\Bm\times\Bm,\HK(\C))$ when $l\rightarrow \infty$. This can be observed from the case $f\in C(\overline{\Bm}\times\overline{\Bm},\HK(\C))$ and the fact that $C(\overline{\Bm}\times\overline{\Bm},\HK(\C))$ is dense in $L^2(\Bm\times\Bm,\HK(\C))$. More specifically, for $f\in L^2(\Bm\times\Bm,\HK(\C))$, there exists a sequence $\{g_s\}_{s=1}^{\infty}\in C(\overline{\Bm}\times\overline{\Bm},\HK(\C))$ such that for arbitrary $\epsilon>0$, there exists $N>0$, such that $\|f-g_s\|_{L^2(\Bm\times\Bm,\HK(\C))}<\epsilon$ when $s>N$. Let $g_{s,l}(\bx,\bv)=g_s(\frac{l-1}{l}\bx,\bv)$, then we can easily check that $\|f_l-g_{s,l}\|_{L^2(\Bm\times\Bm,\HK(\C))}<\frac{l}{l-1}\epsilon$ when $s>N$. Therefore, we have
\begin{align}\label{3esti}
\|f-f_l\|_{L^2(\Bm\times\Bm,\HK(\C))}\leq& \|f-g_s\|_{L^2(\Bm\times\Bm,\HK(\C))}+\|g_s-g_{s,l}\|_{L^2(\Bm\times\Bm,\HK(\C))}\nonumber\\
&+\|g_{s,l}-f_l\|_{L^2(\Bm\times\Bm,\HK(\C))}
\leq \epsilon+\epsilon+\frac{l}{l-1}\epsilon=\frac{3l-1}{l}\epsilon,
\end{align}
when $l>N'$ and $s>N$, where $N'$ is sufficiently large so that $\|g_s-g_{s,l}\|_{L^2(\Bm\times\Bm,\HK(\C))}<\epsilon$, which comes from the continuity of $g_s$.
\par
Since $\ker\Dtwo$ is invariant with respect to dilations,  any function $f\in b^2(\Bm\times\Bm,\HK(\C))$ can be approximated by a sequence of functions $\{f_l\}_{l=1}^{\infty}$ satisfying $\Dtwo f_l=0$ in $\overline{\Bm}\times\overline{\Bm}$. Further, notice that $f_l\in C^2(\frac{l}{l-1}\Bm\times\Bm,\HK(\C))$, then Proposition \ref{PropConverge} tells us that each $f_l$ can be approximated by a sequence of homogeneous polynomial null solutions to $\Dtwo$, which converges absolutely and uniformly in $\overline{\Bm}\times\overline{\Bm}$. Hence, with a similar argument as applied in (\ref{3esti}), there exists a sequence of homogeneous polynomials null solutions to $\Dtwo$ which converges to $f$ in $L^2(\Bm\times\Bm,\HK(\C))$ as desired.
\par
The claim that the decomposition is an orthogonal decomposition can be observed by changing to polar coordinates and applying Theorem \ref{thmortho}.
\end{proof}
The proposition below provides a series of expansion for the reproducing kernel $R_{k,\Bm\times\Bm}$ in terms of $J_{l,k}$.
\begin{proposition}
Let $\bx,\bv\in\Bm$, then we have
\be
R_{k,\Bm\times\Bm}(\by,\bx,\bu,\bv)=\sum_{l=0}^{\infty}(m+2l)(m+2k)J_{l,k}(\by,\bx,\bu,\bv),
\ee
the series converges absolutely and uniformly in $\Bm\times K\times\Bm\times\Bm$ for all compact subset $K\subset \Bm$.
\end{proposition}
\begin{proof}
For $\bx,\by,\bu,\bv\in\Bm\backslash{\{0\}}$, we have 
\begin{align*}
|J_{l,k}(\by,\bx,\bu,\bv)|&=|\bx\by|^l|\bu\bv|^k|J_{l,k}(\frac{\by}{|\by|},\frac{\bx}{|\bx|},\frac{\bu}{|\bu|},\frac{\bv}{|\bv|})|\leq|\bx\by|^l|\bu\bv|^k|\dim\HK\dim\mathcal{H}_l\\
&\leq C_m(lk)^{m-2}|\bx|^l|\by|^l|\bu|^k|\bv|^k.
\end{align*}
This implies that $\sum_{l=0}^{\infty}(m+2l)(m+2k)J_{l,k}(\by,\bx,\bu,\bv)$ converges absolutely and uniformly in $\Bm\times K\times\Bm\times\Bm$, where $K$ is a compact subset of $\Bm$. Now, if we denote $F(\by,\bx,\bu,\bv)=\sum_{l=0}^{\infty}(m+2l)(m+2k)J_{l,k}(\by,\bx,\bu,\bv)$, then $F(\cdot,\bx,\cdot,\bv)\in\ker\Dtwo$ is bounded in $\Bm\times\Bm$ for fixed $\bx,\bv\in\Bm$ and hence $F(\cdot,\bx,\cdot,\bv)\in b^2(\Bm\times\Bm,\HK(\C))$. Now, if $f$ is a polynomial solution to $\Dtwo$, then with the calculation in \eqref{BtoS} and the orthogonality given in Proposition \ref{thmortho}, we can easily obtain that $\langle F(\cdot,\bx,\cdot,\bv),f \rangle_{b^2}=f$ in $\Bm\times\Bm$. Further, the previous proposition tells us that a function in $b^2(\Bm\times\Bm,\HK(\C)$ can be approximated by polynomial solutions to $\Dtwo$. This immediately gives us that for any $f\in b^2(\Bm\times\Bm,\HK(\C)$, we have $\langle F(\cdot,\bx,\cdot,\bv),f \rangle_{b^2}=f$ in $\Bm\times\Bm$. This implies that $F(\cdot,\bx,\cdot,\bv)$ is the reproducing kernel of $b^2(\Bm\times\Bm,\HK(\C)$. which completes the proof.
\end{proof}
The next result provides an expression for the bosonic Bergman projection in terms of the reproducing kernel $J_{l,k}$.
\begin{proposition}
Let $f\in\mathcal{P}_s(\Bm\times\Bm,\HK(\C))$. Then $\deg_{\bx}B_{k,\Bm\times\Bm}[f]\leq s$, and 
\be
B_{k,\Bm\times\Bm}[f](\bx,\bv)=(m+2k)\sum_{l=0}^s(m+2l)\int_{\Bm}\int_{\Bm}J_{l,k}(\by,\bx,\bu,\bv)f(\by,\bu)d\bu d\by,
\ee
for all $\bx,\bv\in\Bm$.
\end{proposition}
\begin{proof}
This result can be observed from the previous proposition and Proposition \ref{BergmanProjection}. The orthogonality given in Corollary \ref{CorOrtho} and the decomposition given in Theorem \ref{decomp} explain the disappearance of the terms in the series with $l>s$.
\end{proof}
The following corollary tells us the connection between the bosonic Bergman projection of a function in $\mathcal{P}_{s}(\Bm\times\Bm,\HK(\C))$ and its Poisson integral.
\begin{corollary}
Let $f\in \mathcal{P}_{s}(\Bm\times\Bm,\HK(\C))$ and $\sum_{l=0}^{\infty}f_l$ is the solution to the Dirichlet problem in the unit ball with boundary data $f\big\vert_{\Smone\times\Smone}$, where $f_l\in\BBK$. Then, we have
\be
B_{k,\Bm\times\Bm}[f]=\sum_{l=0}^s\frac{m+2l}{m+l+s}f_l.
\ee
\end{corollary}
\begin{proof}
For $0\leq l\leq s$ and $\bx,\bv\in\Bm$, one has
\begin{align*}
&\int_{\Bm}\int_{\Bm}J_{l,k}(\by,\bx,\bu,\bv)f(\by,\bu)d\bu d\by\\
=&\int_0^1\int_0^1\int_{\Smone}\int_{\Smone}r_1^{m-1}r_2^{m-1}J_{l,k}(r_1\ze,\bx,r_2\et,\bv)f(r_1\ze,r_2\et)dS(\et)dS(\ze)dr_1dr_2\\
=&\int_0^1\int_0^1\int_{\Smone}\int_{\Smone}r_1^{m+l+s-1}r_2^{m+2k-1}J_{l,k}(\ze,\bx,\et,\bv)f(\ze,\et)dS(\et)dS(\ze)dr_1dr_2\\
=&(m+l+s)^{-1}(m+2k)^{-1}\int_{\Smone}\int_{\Smone}J_{l,k}(\ze,\bx,\et,\bv)f(\ze,\et)dS(\et)dS(\ze)\\
=&(m+l+s)^{-1}(m+2k)^{-1}f_l(\bx,\bv),
\end{align*}

where the last equation is obtained by applying the Poisson integral formula \eqref{PI}, the expression of the Poisson kernel given in Proposition \ref{PJ}, the decomposition of $f(\ze,\et)$ given in Theorem \ref{decomp} and the orthogonality given in Theorem \ref{thmortho} . Plugging the equation above into the previous proposition gives us the desired result.
\end{proof}

\section{Hardy spaces related to bosonic Laplacians}
Recall that we defined the Poisson integral of a function $f\in C(\Smone\times\Smone,\HK(\Smone))$ in \cite[(3.13)]{DTR}. In this section, we generalize this definition for a certain complex measure space. Weak$^*$ convergence properties and growth estimates for the Poisson integrals of these complex measures are investigated. Further, the growth estimates lead us to a generalization of harmonic Hardy spaces in higher spin spaces, named as bosonic Hardy spaces. Some growth estimates and characterizations for functions in the bosonic Hardy spaces are also studied here.
\subsection{Poisson integrals of measures and weak$^*$ convergence}
Recall that  the Poisson integral given in \eqref{PI}, we call
\be
P(\ze,\bx,\bu,\bv)=\frac{c_{m,k}}{2}\frac{1-|\bx|^2}{|\bx-\ze|^m}Z_k\bigg(\frac{(\bx-\ze)\bs{u}(\bx-\ze)}{|\bx-\ze|^2},\bs{v}\bigg),\quad \bx,\bv\in\Bm,\ze,\bu\in\Smone
\ee
the Poisson kernel of the bosonic Laplacian in the unit ball. Now, we extend the definition of the Poisson integral above as follows. For a complex measure $\mu=\mu_1\times\mu_2$ on $\Smone\times\Smone$, where $\mu_i$ are finite complex measures on $\Smone$, $i=1,2$, the Poisson integral of $\mu$, denoted by $P[\mu]$, is given by
\begin{eqnarray}\label{PIB}
P[\mu](\bx,\bv):=\int_{\Smone}\int_{\Smone}P(\ze,\bx,\bu,\bv)d\mu_2(\bu)d\mu_1(\ze).
\end{eqnarray}
Further, differentiating under the integral sign above, we can see that $\Dtwo P[\mu]=0$ on $\Bm\times\Bm$.
\par
Let $M(\Smone\times\Smone)$ stand for the set of finite complex Borel measures on $\Smone\times\Smone$. The total variation norm of $\mu\in M(\Smone\times\Smone)$ is denoted by $\|\mu\|$. Since $M(\Smone\times\Smone)$ is a Banach space under the total variation norm, the Riesz Representation Theorem tells us that $M(\Smone\times\Smone)$ is isometrically isomorphic to the dual space of $C(\Smone\times\Smone)$ with the following identification
\begin{eqnarray}\label{iden}
M(\Smone\times\Smone) &\longrightarrow& C(\Smone\times\Smone)^*,\nonumber\\
\mu &\mapsto& \Lambda_{\mu},
\end{eqnarray}
where
\be
\Lambda_{\mu}(f)=\int_{\Smone}\int_{\Smone}fd\mu,\quad \text{for}\ f\in C(\Smone\times\Smone).
\ee
Let $L^p(\Smone\times\Smone)$, $1\leq p<\infty$ be the space of the Borel measurable functions $f$ on $\Smone\times\Smone$ for which
\be
\|f\|_p^p=\int_{\Smone}\int_{\Smone}|f(\bx,\bu)|^pdS(\bx)dS(\bu)<+\infty.
\ee
$L^{\infty}(\Smone\times\Smone)$ consists of the Borel measurable functions $f$ on $\Smone\times\Smone$ for which $\|f\|_{\infty}<+\infty$, where $\|f\|_{\infty}$ stands for the essential supremum norm on $\Smone\times\Smone$ with respect to $dS\times dS$.
\par
Let $C(\Smone\times\Smone,\HK(\C))$ stand for functions $f(\bx,\bu)\in C(\Smone\times\Smone)$ and for each fixed $\bx\in\Smone$, $f(\bx,\bu)\in \HK(\C)$ in the variable $\bu$. A similar definition applies to $L^p(\Smone\times\Smone,\HK(\C))$. We also define $M(\Smone\times\Smone,\HK(\C))$ to be the subspace of the space of finite complex measures on $\Smone\times\Smone$, which satisfies that for each $\mu=\mu_1\times\mu_2\in M(\Smone\times\Smone,\HK(\C))$, then $\mu\in M(\Smone\times\Smone)$ and $d\mu_2=hdS$, where $h$ is a $k$-homogeneous harmonic polynomial. We claim that the identification given in \eqref{iden} also provides an isometrical isomorphism between $M(\Smone\times\Smone,\HK(\C))$ and $C(\Smone\times\Smone,\HK(\C))^*$, which is stated as follows.
\begin{lemma}
$M(\Smone\times\Smone,\HK(\C))$  is isometrically isomorphic to the dual space of $C(\Smone\times\Smone,\HK(\C))$ with the identification
\be
M(\Smone\times\Smone,\HK(\C)) &\longrightarrow& C(\Smone\times\Smone,\HK(\C))^*,\nonumber\\
\mu &\mapsto& \Lambda_{\mu},
\ee
where
\be
\Lambda_{\mu}(f)=\int_{\Smone}\int_{\Smone}fd\mu,\quad \text{for}\ f\in C(\Smone\times\Smone,\HK(\C)).
\ee
\end{lemma}
\begin{proof}
Indeed, the difference between the isometrical isomorphism given in \eqref{iden} and the isomorphism above is the extra condition $\HK(\C)$ added to the second variable $\bu$. Therefore, to prove the lemma above, we only need to show that for the variable $\bu$,
\begin{eqnarray}\label{linearmap}
M(\Smone,\HK) &\longrightarrow& \HK(\C)^*,\nonumber\\
\mu &\mapsto& \Lambda_{\mu},\quad \bigg(\Lambda_{\mu}(f)=\int_{\Smone}fd\mu,\quad \text{for}\ f\in \HK(\C)\bigg),
\end{eqnarray}
is an isometrical isomorphism, where $M(\Smone,\HK(\C))$ stands for the space of finite complex Borel measures given by $d\mu=hdS$, and $h$ is a $k$-homogeneous harmonic polynomial. To prove this, it suffices to show the linear map \eqref{linearmap} is into, because $\dim M(\Smone,\HK(\C))=\dim \HK(\C)^*=\dim\HK(\C)$ is finite. This is equivalent to prove that if $\Lambda_{\mu}(f)=0$ for all $f\in\HK(\C)$, then $\mu=0$. Assume that $\mu=hdS$, where $h\in\HK(\C)$, then 
\be
\Lambda_{\mu}(f)=\int_{\Smone}fd\mu=\int_{\Smone}fhdS.
\ee
Since $f\in\HK(\C)$ is arbitrary, then $h=0$, which means $\mu=0$. This completes the proof for \eqref{linearmap}. Hence, our lemma is true.
\end{proof}
\par
When given a function $f$ on $\Bm\times\Bm$, the notation $f_{r_1,r_2}$ stands for the function on $\Smone\times\Smone$ defined by $f_{r_1,r_2}(\bs{\xi},\bs{\eta})=f(r_1\bs{\xi},r_2\bs{\eta})$, where $0\leq r_1,r_2<1$. Next, we introduce growth estimates for the Poisson integrals of measures.
\begin{theorem}\label{growth1}
The following estimates apply to Poisson integrals.
\begin{enumerate}
\item If $\mu=\mu_1\times\mu_2\in M(\Smone\times\Smone,\HK(\C))$. Further, let $f=P[\mu]$, then $\|f_{r_1,r_2}\|_1\leq \frac{m+2k-2}{m-2}\|\mu\|$ for all $r_1,r_2\in[0,1)$.
\item If $1\leq p\leq \infty$, $g\in L^p(\Smone\times\Smone,\HK(\C))$ and $f=P[g]$, then $\|f_{r_1,r_2}\|_p\leq \frac{m+2k-2}{m-2} \|g\|_p$ for all $r_1,r_2\in[0,1)$.
\end{enumerate}
\end{theorem}
\begin{remark}
One can easily check that when $k=0$, the results above reduce to the properties of Laplacian $\Delta_{\bx}$.
\end{remark}
\begin{proof}
To prove $(1)$, firstly, we have
\be
&&f_{r_1,r_2}(\bs{\xi},\bs{\eta})=f(r_1\bs{\xi},r_2\bs{\eta})=P[\mu](r_1\bs{\xi},r_2\bs{\eta})\\
&=&\frac{c_{m,k}}{2}\int_{\Smone}\int_{\Smone}\frac{1-|r_1\bs{\xi}|^2}{|r_1\bs{\xi}-\ze|^m}Z_k\bigg(\frac{(r_1\bs{\xi}-\ze)\bs{u}(r_1\bs{\xi}-\ze)}{|r_1\bs{\xi}-\ze|^2},r_2\et\bigg)h(\bu)dS(\bu)d\mu_1(\ze)\\
&=&\frac{c_{m,k}}{2}\int_{\Smone}\int_{\Smone}\frac{1-|r_1\bs{\xi}|^2}{|r_1\bs{\xi}-\ze|^m}Z_k\big(\bu,\phi(\bs{\eta})\big)h(\bu)dS(\bu)d\mu_1(\ze)\\
&=&\omega_m^{-1}\int_{\Smone}\frac{1-|r_1\bs{\xi}|^2}{|r_1\bs{\xi}-\ze|^m}h(\phi(\bs{\eta}))d\mu_1(\ze),
\ee
where
$
\phi(\bs{\eta})=\frac{(r_1\bs{\xi}-\ze)r_2\et(r_1\bs{\xi}-\ze)}{|r_1\bs{\xi}-\ze|^2},
$
and $\omega_m$ is the area of $\Smone$. The last equation above is obtained from the facts that $Z_k(\bs{a}\bu \bs{a},\bv)=Z_k(\bu,\bs{a}\bv \bs{a})$ with $\bs{a}\in\R^m$ and $Z_k$ is the reproducing kernel of $k$-homogeneous harmonic polynomials and  \cite[Lemma 3.6]{DTR}.
\par
Now, we have
\be
\|f_{r_1,r_2}\|_{1}&=&\frac{c_{m,k}}{2}\int_{\Smone}\int_{\Smone}\bigg\vert\int_{\Smone}\frac{1-|r_1\bs{\xi}|^2}{|r_1\bs{\xi}-\ze|^m}h(\phi(\bs{\eta}))d\mu_1(\ze)\bigg\vert dS(\bs{\xi})dS(\bs{\eta})\\
&\leq&\frac{c_{m,k}}{2}\int_{\Smone}\int_{\Smone}\frac{1-|r_1\bs{\xi}|^2}{|r_1\bs{\xi}-\ze|^m}\int_{\Smone}|h(\phi(\bs{\eta}))|dS(\bs{\eta})dS(\bs{\xi})d|\mu_1|(\ze).
\ee
Further, noticing that $\phi(\bs{\eta})$ is a reflection of $r_2\bs{\eta}$, hence, the integral above becomes
\begin{align*}
&\frac{c_{m,k}}{2}\int_{\Smone}\int_{\Smone}\frac{1-|r_1\bs{\xi}|^2}{|r_1\bs{\xi}-\ze|^m}\int_{\Smone}|h(r_2\bs{\eta})|dS(\bs{\eta})dS(\bs{\xi})d|\mu_1|(\ze)\\
&=\frac{m+2k-2}{m-2}\int_{\Smone}\int_{\Smone}|h(r_2\bs{\eta})|dS(\bs{\eta})d|\mu_1|(\ze)
\leq \frac{m+2k-2}{m-2}\int_{\Smone}\int_{\Smone}|h(\bs{\eta})|dS(\bs{\eta})d|\mu_1|(\ze)\\
&=\frac{m+2k-2}{m-2}\int_{\Smone}\int_{\Smone}d|\mu_2|(\bs{\eta})d|\mu_1|(\ze)\leq\frac{m+2k-2}{m-2}\|\mu\|,
\end{align*}
where the following facts are used in the last two steps above
\be
\omega_m^{-1}\int_{\Smone}\frac{1-|r_1\bs{\xi}|^2}{|r_1\bs{\xi}-\ze|^m}dS(\bs{\xi})=\omega_m^{-1}\int_{\Smone}\frac{1-|r_1|^2}{|\bs{\xi}-r_1\ze|^m}dS(\bs{\xi})=1,
\ee
and 
$
\int_{\Smone}|h(r_2\bs{\eta})|dS(\bs{\eta})
$
 is increasing with respect to $r_2$, since $h$ is harmonic, see \cite[Corollary 6.6]{Axler}.
\par
For $(2)$, we firstly assume that $1\leq p<\infty$, with similar argument applied to $(1)$, we obtain
\begin{align*}
|f_{r_1,r_2}(\bs{\xi},\bs{\eta})|
\leq\frac{c_{m,k}}{2}\int_{\Smone}\frac{1-|r_1\bs{\xi}|^2}{|r_1\bs{\xi}-\ze|^m}|g(\bs{\zeta},\phi(\bs{\eta}))|dS(\ze).
\end{align*}
Then, with the Jensen's integral inequality, we have
\begin{align*}
&\bigg(\int_{\Smone}\int_{\Smone}|f_{r_1,r_2}(\bs{\xi},\bs{\eta})|^pdS(\bs{\xi})dS(\bs{\eta})\bigg)^{1/p}\\
\leq&\frac{c_{m,k}}{2}\bigg(\int_{\Smone}\int_{\Smone}\int_{\Smone}\frac{1-|r_1\bs{\xi}|^2}{|r_1\bs{\xi}-\ze|^m}|g(\bs{\zeta},\phi(\bs{\eta}))|^pdS(\ze)dS(\bs{\xi})dS(\bs{\eta})\bigg)^{1/p}\\
=&\frac{c_{m,k}}{2}\bigg(\int_{\Smone}\int_{\Smone}\int_{\Smone}\frac{1-|r_1\bs{\xi}|^2}{|r_1\bs{\xi}-\ze|^m}|g(\bs{\zeta},\bs{\eta})|^pdS(\bs{\eta})dS(\ze)dS(\bs{\xi})\bigg)^{1/p}\\
=&\frac{m+2k-2}{m-2}\bigg(\int_{\Smone}\int_{\Smone}|g(\bs{\zeta},\bs{\eta})|^pdS(\bs{\eta})dS(\ze)\bigg)^{1/p}
=\frac{m+2k-2}{m-2}\|g\|_p.
\end{align*}
Similar argument as above can be applied to $p=\infty$.
\end{proof}
An immediately consequence of  the theorem above is the following.
\begin{corollary}
Let $1\leq p<\infty$, $g\in L^p(\Smone\times\Smone,\HK(\C))$ and $f=P[g]$. Then,
\be
\|f_{r_1,r_2}\|_p\leq \frac{m+2k-2}{m-2}\|g_{s_1,s_2}\|_p
\ee
for all $0\leq r_j\leq s_j\leq1$, $j=1,2$.
\end{corollary}
\begin{proof}
This can be obtained immediately from the previous theorem and \cite[Theorem 3.7]{DTR} as follows
\be
\|f_{r_1,r_2}\|_p=\|P[g_{s_1,s_2}]_{\frac{r_1}{s_1},\frac{r_2}{s_2}}\|_p\leq\frac{m+2k-2}{m-2}\|g_{s_1,s_2}\|_p.
\ee
\end{proof}
Recall that \cite[Theorem 3.7]{DTR} tells us that if $g\in C^2(\Smone\times\Bm,\HK(\C))$, and $f=P[g]$, then $f_{r,1}\rightarrow g$ in $C(\Smone\times\Bm)$ as $r\rightarrow 1$. Actually, with the previous theorem, we have a generalized result of \cite[Theorem 3.10]{DTR} as follows.
\begin{corollary}\label{corlp}
Suppose $1\leq p<\infty$. If $g\in L^p(\Smone\times\Smone,\HK(\C))$ and $f=P[g]$, then we have $\|f_{r_1,r_2}-g\|_p\rightarrow 0$ as $r_1,r_2\rightarrow 1$.
\end{corollary}
\begin{proof}
Suppose $1\leq p<\infty$ and a function $g\in L^p(\Smone\times\Smone,\HK(\C))$, let $f=P[g]$. For a fixed $\epsilon>0$, there exists a function $h\in C(\Smone\times\Smone,\HK(\C))$, such that $\|g-h\|_p<\epsilon$. Let $f^{\dagger}=P[h]$, then
$
\|f_{r_1,r_2}-g\|_p\leq\|f_{r_1,r_2}-f^{\dagger}_{r_1,r_2}\|_p+\|f^{\dagger}_{r_1,r_2}-h\|_p+\|h-g\|_p.
$
According to Theorem \ref{growth1}, we observe that 
\be
\|f_{r_1,r_2}-f^{\dagger}_{r_1,r_2}\|_p=\|P[g-h]_{r_1,r_2}\|_p\leq\frac{m+2k-2}{m-2}\|g-h\|_p\leq\frac{m+2k-2}{m-2}\epsilon.
\ee
Further, with \cite[Theorem 3.7]{DTR} and \cite[Theorem 6.7]{Axler}, we obtain $\|f^{\dagger}_{r_1,r_2}-h\|_p\leq \epsilon$ when $r_1,r_2\rightarrow 0$. Therefore, we have 
$
\|f_{r_1,r_2}-f^{\dagger}_{r_1,r_2}\|_p\leq \big(\frac{m+2k-2}{m-2}+2\big)\epsilon.
$
Since $\epsilon$ is arbitrary, we complete the proof.
\end{proof}

\begin{theorem}\label{weak}
Poisson integrals also have the following weak$^*$ convergence properties.
\begin{enumerate}
\item If $\mu=\mu_1\times\mu_2\in M(\Smone\times\Smone,\HK(\C))$ and $f=P[\mu]$, then $f_{r_1,r_2}\rightarrow \mu$ weak$^*$ in $M(\Smone\times\Smone,\HK(\C))$ as $r_1,r_2\rightarrow 1$.
\item Let $1\leq p\leq\infty$, if $g\in L^{p}(\Smone\times\Smone,\HK(\C))$ and $f=P[g]$, then $f_{r_1,r_2}\rightarrow g$ weak$^*$ in $L^{p}(\Smone\times\Smone,\HK(\C))$ as $r_1,r_2\rightarrow 1$.
\end{enumerate}
\end{theorem}
\begin{proof}
Recall that $M(\Smone\times\Smone,\HK(\C))=C(\Smone\times\Smone,\HK(\C))^*$, let $f=P[\mu]$ with $\mu=\mu_1\times\mu_2\in M(\Smone\times\Smone,\HK(\C))$. To prove $(1)$, we only need to show that for all $g\in C(\Smone\times\Smone,\HK(\C))$,
\be
\lim_{r_1,r_2\rightarrow 1}\int_{\Smone}\int_{\Smone}f_{r_1,r_2}(\bs{\xi},\bs{\eta})g(\bs{\xi},\bs{\eta})dS(\bs{\xi})dS(\bs{\eta})= \int_{\Smone}\int_{\Smone}g(\ze,\bu)d\mu_1(\ze)d\mu_2(\bu).
\ee
Now, we look at the integral on the left side, with a similar argument as in $(1)$ in Theorem \ref{growth1},
\be
&&\lim_{r_1,r_2\rightarrow 1}\int_{\Smone}\int_{\Smone}f_{r_1,r_2}(\bs{\xi},\bs{\eta})g(\bs{\xi},\bs{\eta})dS(\bs{\xi})dS(\bs{\eta})\\
&=&\lim_{r_1,r_2\rightarrow 1}\int_{\Smone}\int_{\Smone}\int_{\Smone}\int_{\Smone}P(\ze,r_1\bs{\xi},\bs{u},r_2\et)d\mu_2(\bu)d\mu_1(\ze)g(\bs{\xi},\bs{\eta})dS(\bs{\xi})dS(\bs{\eta}).
\ee
Recall that
$
P(\ze,r_1\bs{\xi},\bs{u},r_2\et)=\frac{c_{m,k}}{2}\frac{1-r_1^2}{|r_1\bs{\xi}-\ze|^m}Z_k\big(\frac{(r_1\bs{\xi}-\ze)\bs{u}(r_1\bs{\xi}-\ze)}{|r_1\bs{\xi}-\ze|^2},r_2\et\big).
$
Since $Z_k(\bu,\bv)$ is a homogeneous polynomial, when $r_1,r_2\rightarrow 1$, the singularities only happen in $\frac{1-r_1^2}{|r_1\bs{\xi}-\ze|^m}$. Therefore, we can let $r_1=r_2=1$ in $Z_k,\ \bs{\omega}$ above. Further, since $\ze,\bs{\eta}\in\Smone$, we have
$
\frac{1-r_1^2}{|r_1\bs{\xi}-\ze|^m}=\frac{1-r_1^2}{|\bs{\xi}-r_1\ze|^m}.
$
Hence,
the integral above becomes
\begin{align*}
&\lim_{r_1,r_2\rightarrow 1}\int_{\Smone}\int_{\Smone}\int_{\Smone}\int_{\Smone}P(r_1\ze,\bs{\xi},\bs{u},r_2\et)g(\bs{\xi},\bs{\eta})dS(\bs{\xi})dS(\bs{\eta})d\mu_2(\bu)d\mu_1(\ze)\\
=&\lim_{r_1\rightarrow 1}\int_{\Smone}\int_{\Smone}\int_{\Smone}\int_{\Smone}P(r_1\ze,\bs{\xi},\bs{u},\et)g(\bs{\xi},\bs{\eta})dS(\bs{\xi})dS(\bs{\eta})d\mu_2(\bu)d\mu_1(\ze)\\
=&\lim_{r_1\rightarrow 1}\int_{\Smone}\int_{\Smone}P[g]_{r_1,1}(\ze,\bu)d\mu_2(\bu)d\mu_1(\ze)\\
=&\lim_{r_1,r_2\rightarrow 1}\int_{\Smone}\int_{\Smone}P[g]_{r_1,r_2}(\ze,\bu)d\mu_2(\bu)d\mu_1(\ze)\\
=&\int_{\Smone}\int_{\Smone}g(\ze,\bu)d\mu_2(\bu)d\mu_1(\ze),
\end{align*}
which completes the proof of $(1)$.
\par
The proof of $(2)$ is similar as in $(1)$. Firstly, we prove the case $1< p<\infty$. We notice that $L^p(\Smone\times\Smone,\HK(\C))^*=L^q(\Smone\times\Smone,\HK(\C))$, where $q$ is the conjugate number of $p$, i.e., $\frac{1}{p}+\frac{1}{q}=1$. Let $g\in L^p(\Smone\times\Smone,\HK(\C))$ and $f=P[g]$, we need to show that for all $h\in L^q(\Smone\times\Smone,\HK(\C))$, 
\be
\lim_{r_1,r_2\rightarrow 1}\int_{\Smone}\int_{\Smone}\big(f_{r_1,r_2}(\bs{\xi},\bs{\eta})-g(\bs{\xi},\bs{\eta})\big)h(\bs{\xi},\bs{\eta})dS(\bs{\xi})dS(\bs{\eta})= 0.
\ee
According to H\"older's inequality, the integral above can be estimated by
\be
\leq\|f_{r_1,r_2}-g\|_p\cdot\|h\|_q\longrightarrow 0,\quad \text{as}\ r_1,r_2\to 1,
\ee
where Corollary \ref{corlp} and $h\in L^q(\Smone\times\Smone,\HK(\C))$ are applied above. This proves $(2)$ when $1<p<\infty$. The argument is the same for the case $p=1$ except that we do not need to use H\"older's inequality.
\par
For $p=\infty$, we notice that $L^1(\Smone\times\Smone,\HK(\C))^*=L^{\infty}(\Smone\times\Smone,\HK(\C))$. With $g\in L^{\infty}(\Smone\times\Smone,\HK(\C))$ and $f=P[g]$, we need to show that for all $h\in L^1(\Smone\times\Smone,\HK(\C))$,
\be
\lim_{r_1,r_2\rightarrow 1}\int_{\Smone}\int_{\Smone}f_{r_1,r_2}(\bs{\xi},\bs{\eta})h(\bs{\xi},\bs{\eta})dS(\bs{\xi})dS(\bs{\eta})= \int_{\Smone}\int_{\Smone}g(\bs{\xi},\bs{\eta})h(\bs{\xi},\bs{\eta})dS(\bs{\xi})dS(\bs{\eta}).
\ee
With a similar argument as in $(1)$, the left hand side of the equation above becomes
\be
\lim_{r_1,r_2\rightarrow 1}\int_{\Smone}\int_{\Smone}P[g]_{r_1,r_2}(\ze,\bu)h(\ze,\bu)dS(\bu)dS(\ze).
\ee
The Corollary \ref{corlp} tells us that $P[g]_{r_1,r_2}(\ze,\bu)$ converges to $g(\ze,\bu)$ in $L^1(\Smone\times\Smone,\HK(\C))$ as $r_1,r_2\rightarrow 1$, since $g\in L^{\infty}(\Smone\times\Smone,\HK(\C))$, we also have that $P[g]_{r_1,r_2}(\ze,\bu)h(\ze,\bu)$ converges to $g(\ze,\bu)h(\ze,\bu)$ in $L^1(\Smone\times\Smone,\HK(\C))$ as $r_1,r_2\rightarrow 1$. This completes the proof of the theorem.
\end{proof}
Recall that in \cite{DTR}, for a function $f\in L^p(\Bm\times\Bm,\HK(\C))$ with $1\leq p<\infty$, we define the norm
\be
\|f\|_{L^p(\Bm\times\Bm,\HK(\C))}:=\bigg(\int_{\Bm}\int_{\Smone}|f(\bx,\bu)|^pdS(\bu)d\bx\bigg)^{1/p}.
\ee
We can also obtain a growth estimate of the $L^p$-norm of the Poisson integral of the measure space $M(\Smone\times\Smone,\HK(\C))$.
\begin{proposition}\label{PropPu}
Let $\mu\in M(\Smone\times\Smone,\HK(\C))$ and $1\leq p<\frac{m}{m-1}$. Then, there exists a constant $\lambda_{m,k}$, which only depends on $m$ and $k$, such that
\be
\|P[\mu]\|_{L^p(\Bm\times\Bm,\HK(\C))}\leq \lambda_{m,k}\|\mu\|.
\ee
\end{proposition}
\begin{proof}
Let $\mu=\mu_1\times\mu_2\in M(\Smone\times\Smone,\HK(\C))$, $p>1$ and $q$ is the conjugate number of $p$. For convenience, all $\lambda_{m,k}$ in this proof are positive constants only depending on $m$ and $k$ but not necessarily the same. We have
\be
&&|P[\mu](\bx,\bv)|=\bigg\vert\int_{\Smone}\int_{\Smone}P(\ze,\bx,\bu,\bv)d\mu_2(\bu)d\mu_1(\ze)\bigg\vert\\
&=&\bigg\vert\int_{\Smone}\int_{\Smone}\frac{c_{m,k}}{2}\frac{1-|\bx|^2}{|\bx-\ze|^m}Z_k\bigg(\frac{(\bx-\ze)\bs{u}(\bx-\ze)}{|\bx-\ze|^2},\bs{v}\bigg)d\mu_2(\bu)d\mu_1(\ze)\bigg\vert\\
&\leq&\lambda_{m,k}\bigg\vert\int_{\Smone}\int_{\Smone}\frac{1-|\bx|^2}{|\bx-\ze|^m}d|\mu_2|(\bu)d|\mu_1|(\ze)\bigg\vert\\
&\leq&\lambda_{m,k}\bigg[\int_{\Smone}\int_{\Smone}\bigg(\frac{1-|\bx|^2}{|\bx-\ze|^m}\bigg)^pd|\mu_2|(\bu)d|\mu_1|(\ze)\bigg]^{\frac{1}{p}}\bigg[\int_{\Smone}\int_{\Smone}d|\mu_2|(\bu)d|\mu_1|(\ze)\bigg]^{\frac{1}{q}},
\ee
where the fact that $|Z_k(\bu,\bv)|\leq\dim\HK(\C)$ for $\bu,\bv\in\Bm$  (\cite[Proposition 5.27]{Axler}) is applied above.
Notice that
\be
\frac{1-|\bx|^2}{|\bx-\ze|^m}\leq \frac{1+|\bx|}{|\bx-\ze|^{m-1}}\leq 2|\bx-\ze|^{1-m},\ \text{for\ all}\ \bx\in\Bm,\ze\in\Smone.
\ee
Therefore, we obtain
$
|P[\mu](\bx,\bv)|\leq\lambda_{m,k}|\bx-\ze|^{1-m}\|\mu\|^{\frac{1}{q}+\frac{1}{p}}=\lambda_{m,k}|\bx-\ze|^{1-m}\|\mu\|.
$
Thus,
\be
\|P[\mu]\|_{L^p(\Bm\times\Bm,\HK(\C))}\leq\lambda_{m,k}\bigg(\int_{\Bm}|\bx-\ze|^{p(1-m)}d\bx\bigg)^{1/p}\|\mu\|\leq\lambda_{m,k}\|\mu\|,
\ee
where the integral in the parentheses above exists if $p(1-m)>-m$, i.e., $p<\frac{m}{m-1}$, which completes the proof for $p>1$. The proof for the case $p=1$ is much easier, one can adapt the same argument without using the H\"older's inequality. 
\end{proof}
With the proposition above, one can easily see that
\begin{corollary}
Suppose $\{\mu_n\}$ strongly converges to  $\mu$ in $M(\Smone\times\Smone,\HK(\C))$ as $n\rightarrow \infty$, and $1\leq p<\frac{m}{m-1}$. Then we have $\{P[\mu_n]\}$ strongly converges to $P[\mu]$ in $L^p(\Bm\times\Bm,\HK(\C))$ as $n\rightarrow \infty$.
\end{corollary}

\subsection{Bosonic Hardy spaces and growth estimates}
The growth estimates in Theorem \ref{growth1} suggest us to define analogs of harmonic Hardy spaces as follows.
\par
Let $1\leq p\leq \infty$, we define \emph{bosonic Hardy spaces}, denoted by $h^p(\Bm\times\Bm,\HK(\C))$, to be the class of functions $f(\bx,\bu)\in C^2(\Bm\times\Bm,\HK(\C))$ such that $\Dtwo f=0$ on $\Bm\times\Bm$ and 
\be
\|f\|_{h^p}=\sup_{0\leq r_1,r_2<1}\|f_{r_1,r_2}\|_p<\infty.
\ee
One can easily verify that $h^p(\Bm\times\Bm,\HK(\C))$ is a normed linear space with the norm $\|\cdot\|_{h^p}$. 
\begin{theorem}\label{hp}
The Poisson integral induces the following bijective maps.
\begin{enumerate}
\item The map $\mu\mapsto P[\mu]$ is a linear bijective map from $M(\Smone\times\Smone,\HK(\C))$ to $h^1(\Bm\times\Bm,\HK(\C))$. Further,
$
\|\mu\|\leq\|P[\mu]\|_{h^1}\leq\frac{m+2k-2}{m-2}\|\mu\|.
$
\item For $1<p\leq \infty$, the map $g\mapsto P[g]$ is a linear bijective map from $L^p(\Smone\times\Smone,\HK(\C))$ to $h^p(\Bm\times\Bm,\HK(\C))$. Further,
$
\|g\|_p\leq\|P[g]\|_{h^p}\leq\frac{m+2k-2}{m-2}\|g\|_p.
$
\end{enumerate}
\end{theorem}
To prove the theorem above, we need to following result on weak$^*$ convergence.
\begin{proposition}[Theorem $6.12$,\cite{Axler}]\label{WeakC}
If $X$ is a separable normed linear space, then every norm-bounded sequence in $X^*$ contains a weak$^*$ convergent subsequence.
\end{proposition}
Now, we are ready to prove Theorem \ref{hp}.
\begin{proof}
$(1)$. Firstly, it is obvious that the map $\mu\mapsto P[\mu]$ is linear. Hence, we only need to prove it is bijective. We will prove this by showing it is an injective and surjective map, respectively. To prove it is an injective map, we only need to prove the inequalities in $(1)$. Let $f\in h^1(\Bm\times\Bm,\HK(\C))$, on the one hand, Theorem \ref{growth1} already tells us that $\|P[\mu]\|_{h^1}\leq\frac{m+2k-2}{m-2}\|\mu\|$. On the other hand, since $\{P[\mu]_{r_1,r_2}\}$ converges weak$^*$ to $\mu$ as in Theorem \ref{weak}, we have 
\be
\|\mu\|\leq\underset{r_1,r_2\rightarrow 1}{\lim\inf}\|P[\mu]_{r_1,r_2}\|_1\leq\sup_{0\leq r_1,r_2<1}\|P[\mu]_{r_1,r_2}\|_1=\|P[\mu]\|_{h^1},
\ee
which completes the proof of the inequalities in $(1)$. Hence, the map in $(1)$ is injective.
\par
To prove the map is surjective, let $f\in h^1(\Bm\times\Bm,\HK(\C))$, according to the definition, there exists a family of functions $\{f_{r_1,r_2},r_1,r_2\in[0,1)\}$, which is norm-bounded in $L^1(\Smone\times\Smone,\HK(\C))$, and hence in $M(\Smone\times\Smone,\HK(\C))=C(\Smone\times\Smone,\HK(\C))^*$. Further, since $M(\Smone\times\Smone)$  is a separable metric space and $M(\Smone\times\Smone,\HK(\C))$ is a subspace of $M(\Smone\times\Smone)$. This implies that $M(\Smone\times\Smone,\HK(\C))$ is also separable. Therefore, we have that $\{f_{r_1,r_2},r_1,r_2\in[0,1)\}$ is a norm-bounded sequence in a separable normed linear space $M(\Smone\times\Smone,\HK(\C))$. According to Proposition \ref{WeakC}, we know that there exists a subsequence converges weak$^*$ to some $\mu\in M(\Smone\times\Smone,\HK(\C))$. Now, to complete the proof of $(1)$, we only need to show that $f=P_B[\mu]$. Recall that, for fixed $\bx,\bv\in\Bm$, the Poisson integral formula given in \cite{DTR} tells us that
 \begin{equation*} \begin{aligned}
f(r_1\bx,r_2\bv)=P_B[f](r_1\bx,r_2\bv)=\frac{c_{m,k}}{2}\int_{\Smone}\int_{\Smone}P_B(\ze,\bx,\bu,r_2\bv)f(r_1\ze,\bu)dS(\bu)dS(\ze).
\end{aligned} \end{equation*}
Since $P_B(\ze,\bx,\bu,\bv)$ is continuous on $\Smone$ for fixed $\bx,\bv\in\Bm$, then by the continuity of $f$, we let $r_1,r_2\rightarrow 1$ to immediately obtain that 
 \begin{equation*} \label{PH2} \begin{aligned}
f(\bx,\bv)=\lim_{r_1,r_2\rightarrow 1}P_B[f](r_1\bx,r_2\bv)=\frac{c_{m,k}}{2}\int_{\Smone}\int_{\Smone}P_B(\ze,\bx,\bu,\bv)d\mu (\bu,\ze).
\end{aligned} \end{equation*}
In other words, $f=P_B[\mu]$.
\par
$(2)$. The proof of $(2)$ is similar. Let $1<p\leq\infty$, and $f\in h^p(\Bm\times\Bm,\HK(\C))$. To prove the map is injective, one can apply a similar argument as in $(1)$, with the help of Theorem \ref{growth1} and Theorem \ref{weak}, one can obtain the inequalities in $(2)$, which also implies that the map in $(2)$ is injective. To prove it is surjective, let $q$ is the conjugate number of $p$, then $\{f_{r_1,r_2},r_1,r_2\in[0,1)\}$ is norm-bounded in $L^p(\Smone\times\Smone,\HK(\C))=L^q(\Smone\times\Smone,\HK(\C))^*$. Theorem \ref{WeakC} tells us that there exists a subsequence of $\{f_{r_1,r_2},r_1,r_2\in[0,1)\}$, which converges weak$^*$ to some function $g\in L^p(\Smone\times\Smone,\HK(\C))$. To complete the proof of $(2)$, we need to show that $f=P_B[g]$. This can be derived from a similar argument as in $(1)$ if we can prove that $P_B(\cdot,\bx,\cdot,\bv)\in L^q(\Smone\times\Smone,\HK(\C))$. This can be observed from the fact that, for fixed $\bx\in\Bm$, the classical Poisson kernel $\frac{1-|\bx|^2}{|\bx-\ze|^m}\in L^q(\Smone)$ with respect to $\ze$ and the reproducing kernel $Z_k$ is bounded.
\end{proof}
\begin{remark}
Recall that a bosonic Laplacian reduces to the classical Laplacian $\Delta_{\bx}$ when $k=0$. In this case, the constant $\frac{m+2k-2}{m-2}=1$ in the theorem above. Hence, the theorem above also reduces to some properties of Laplacian. For instance, when $k=0$, $(1)$ states that the classical Poisson integral is a linear bijective isometry from the complex Borel measure space $M(\Smone)$ to the harmonic Hardy space $h^1(\Bm)$.
\end{remark}
Theorem \ref{hp} immediately gives us the following characterizations of the bosonic Hardy spaces.
\begin{proposition}[Characterization of $h^p(\Bm\times\Bm,\HK(\C))$]
\par~
\begin{enumerate}
\item Let $f\in h^p(\Bm\times\Bm,\HK(\C))$ and $1<p\leq\infty$. Then there exists a unique $g\in L^p(\Smone\times\Smone,\HK(\C))$ such that $f=P[g]$. Moreover, 
$
\|g\|_p\leq\|f\|_{h^p}\leq\frac{m+2k-2}{m-2}\|g\|_p.
$
\item Let $f\in h^1(\Bm\times\Bm,\HK(\C))$. Then there exists a unique $\mu\in M(\Smone\times\Smone,\HK(\C))$ such that $f=P[\mu]$. Moreover, 
$
\|\mu\|\leq\|f\|_{h^1}\leq\frac{m+2k-2}{m-2}\|\mu\|.
$
\end{enumerate}
\end{proposition}
Next, we have a growth estimate for functions in $h^p(\Bm\times\Bm,\HK(\C))$.
\begin{proposition}\label{pointestimate}
Suppose $1\leq p\leq \infty$ and $f\in h^p(\Bm\times\Bm,\HK(\C))$, then we have
\be
|f(\bx,\bv)|\leq c_{m,k,p}\bigg(\frac{1+|\bx|}{(1-|\bx|)^{m-1}}\bigg)^{\frac{1}{p}}\|f\|_{h^p},
\ee
for all $\bx\in\Bm$ and $\bv\in\overline{\Bm}$, where
$
c_{m,k,p}=\frac{m+2k-2}{m-2}(\omega_m)^{\frac{p-2}{p}}\dim\HK(\C).
$
\end{proposition}
\begin{proof}
We firstly consider the case $1<p<\infty$. Suppose $f\in h^p(\Bm\times\Bm,\HK(\C))$, then the previous theorem tells us that there exists a function $g\in L^p(\Smone\times\Smone,\HK(\C))$ such that $f=P[g]$. Let $q$ be the conjugate number of $p$, then we have
\begin{eqnarray}\label{eqn1}
&&|f(\bx,\bv)|=\bigg\vert\int_{\Smone}\int_{\Smone}P(\ze,\bx,\bu,\bv)g(\ze,\bu)dS(\bu)dS(\ze)\bigg\vert\nonumber\\
&\leq&\bigg(\int_{\Smone}\int_{\Smone}|P(\ze,\bx,\bu,\bv)|^qdS(\bu)dS(\ze)\bigg)^{\frac{1}{q}}\bigg(\int_{\Smone}\int_{\Smone}|g(\ze,\bu)|^pdS(\bu)dS(\ze)\bigg)^{\frac{1}{p}}\nonumber\\
&=&\bigg(\int_{\Smone}\int_{\Smone}|P(\ze,\bx,\bu,\bv)|^qdS(\bu)dS(\ze)\bigg)^{\frac{1}{q}}\|g\|_p.
\end{eqnarray}
Now, recall that
$
P(\ze,\bx,\bu,\bv)=\frac{c_{m,k}}{2}\frac{1-|\bx|^2}{|\bx-\ze|^m}Z_k\bigg(\frac{(\bx-\ze)\bs{u}(\bx-\ze)}{|\bx-\ze|^2},\bs{v}\bigg),\ \bx,\bv\in\Bm,\ze,\bu\in\Smone,
$
and $Z_k(\bu,\bv)\leq\dim\HK(\C)$ for $\bu,\bv\in\overline{\Bm}$, see \cite[Proposition 5.27]{Axler}. Then,
\begin{eqnarray}\label{estimate}
&&\bigg(\int_{\Smone}\int_{\Smone}|P(\ze,\bx,\bu,\bv)|^qdS(\bu)dS(\ze)\bigg)^{\frac{1}{q}}\nonumber\\
&\leq&\frac{c_{m,k}}{2}\dim\HK(\C)\bigg[\int_{\Smone}\int_{\Smone}\bigg(\frac{1-|\bx|^2}{|\bx-\ze|^m}\bigg)^qdS(\bu)dS(\ze)\bigg]^{\frac{1}{q}}\nonumber\\
&\leq&\frac{c_{m,k}}{2}\dim\HK(\C)\sup_{\ze\in\Smone}\bigg(\frac{1-|\bx|^2}{|\bx-\ze|^m}\bigg)^{\frac{q-1}{q}}\bigg(\int_{\Smone}\int_{\Smone}\frac{1-|\bx|^2}{|\bx-\ze|^m}dS(\bu)dS(\ze)\bigg)^{\frac{1}{q}}\nonumber\\
&=&\frac{m+2k-2}{m-2}(\omega_m)^{\frac{p-2}{p}}\dim\HK(\C)\bigg(\frac{1+|\bx|}{(1-|\bx|)^{m-1}}\bigg)^{\frac{1}{p}}.
\end{eqnarray}
Further, Theorem \ref{growth1} tells us that $\|P[g]_{r_1,r_2}\|_p\leq\frac{m+2k-2}{m-2}\|g\|_p$ and we also know that $\lim_{r_1,r_2\rightarrow 1}P[g]_{r_1,r_2}=P[g]=g$ on $\Smone\times\Smone$. This implies that
\be
\|g\|_p\leq \|P[g]\|_{h^p}=\sup_{0\leq r_1,r_2<1}\|P[g]_{r_1,r_2}\|_p\leq \frac{m+2k-2}{m-2}\|g\|_p,
\ee
in other words, with $f=P[g]$, it is equivalent to 
$
\frac{m-2}{m+2k-2}\|f\|_{h^p}\leq\|g\|_p\leq\|f\|_{h^p}.
$
Plugging this together with \eqref{estimate} into \eqref{eqn1} completes our proof immediately.
\end{proof}
\begin{remark}
This growth estimate also reduces to the one for Laplacian when $k=0$. More specifically, when considering the real-valued case with $k=0$, $\frac{m+2k-2}{m-2}=1$, $\dim\HK=1$ and the $\omega_m$ term will disappear if we normalize our area element $dS$, which eventually gives us  Proposition $6.16$ in \cite{Axler}.
\end{remark}

With the proposition above, we can obtain an estimate for the $L^p$-norm in terms of the $h^p$ norm.
\begin{corollary}
Let $1\leq p<\frac{m}{m-1}$ and $f\in h^p(\Bm\times\Bm,\HK(\C))$, then we have
\be
\|f\|_{L^p(\Bm\times\Bm,\HK(\C))}\leq C_{m,k,p}\|f\|_{h^p}.
\ee
\end{corollary}
\begin{proof}
From the proof of the proposition above, we have that
\be
&&\|f\|_{L^p(\Bm\times\Bm,\HK(\C))}^p\leq C_{m,k,p}\int_{\Bm}\int_{\Bm}\int_{\Smone}|\bx-\ze|^{p(1-m)}dS(\ze)d\bx d\bv\|f\|_{h^p}^p\\
&\leq&C'_{m,k,p}\int_{\Bm}\int_{\Smone}|\bx-\ze|^{p(1-m)}dS(\ze)d\bx \|f\|_{h^p}^p\leq C''_{m,k,p}\|f\|^p_{h^p},
\ee
when $1\leq p<\frac{m}{m-1}$. This completes the proof.
\end{proof}
Proposition \ref{pointestimate} also tells us the following result, which implies that $h^p(\Bm\times\Bm,\HK(\C))$ is a Banach space, in particular, it is a Hilbert space when $p=2$.
\begin{proposition}\label{hpbanach}
The bosonic Hardy space $h^p(\Bm\times\Bm,\HK(\mathbb{C}))$ is a closed subspace of $L^p(\Bm\times\Bm,\mathbb{C})$.
\end{proposition}
\begin{proof}
Suppose that $\{f_j\}_{j=1}^{\infty}$ converges to $f$ in $L^p(\Bm\times\Bm,\mathbb{C})$ and $\{f_j\}_{j=1}^{\infty}$ is a Cauchy sequence in $h^p(\Bm\times\Bm,\HK(\mathbb{C}))$. We will show that $f\in h^p(\Bm\times\Bm,\HK(\mathbb{C}))$ up to a modification on a set of measure zero on $\Bm\times\Bm$.
\par
Let $K_1\times K_2\subset\Bm\times\Bm$ be a compact subset. Proposition \ref{pointestimate} tells us that there exists a constant $C>0$ only depending on $m,k$ and $p$ such that
$
|f_j(\bx,\bv)-f_i(\bx,\bv)|\leq C\|f_j-f_i\|_{h^p},
$
for all $(\bx,\bv)\in K_1\times K_2$ and all $j,i$. Since $\{f_j\}_{j=1}^{\infty}$ is a Cauchy sequence in $h^p(\Bm\times\Bm,\HK(\mathbb{C}))$, $\{f_j\}_{j=1}^{\infty}$ is also a Cauchy sequence in $C(K_1\times K_2)$. Hence, $\{f_j\}_{j=1}^{\infty}$ converges uniformly on $K_1\times K_2$. According to \cite[Proposition 5.8]{DTR}, $\{f_j\}_{j=1}^{\infty}$ converges uniformly to a function $g\in C^2(\Bm\times\Bm,\HK(\mathbb{C}))$ and $\Dtwo g=0$ on $\Bm\times\Bm$. Since $\{f_j\}_{j=1}^{\infty}$ converges to $f$ in $L^p(\Bm\times\Bm,\mathbb{C})$, some subsequence of $\{f_{j}\}_{=1}^{\infty}$ converges to $f$ pointwise almost everywhere on $\Bm\times\Bm$. Therefore, $f=g$ almost everywhere on $\Bm\times\Bm$ and thus $f\in h^p(\Bm\times\Bm,\HK(\mathbb{C}))$, which completes the proof.
\end{proof}

\end{document}